\documentclass[10pt]{siamltex}
\usepackage{geometry}                
\usepackage{amsmath}
\usepackage{graphicx}
\usepackage{url}
\usepackage{a4wide}
\usepackage{color}
\usepackage{amssymb}
\usepackage{epstopdf}
\usepackage{color}
\newcommand{\adots}{\mbox{\setlength{\unitlength}{1pt}
                          \begin{picture}(8,7)\put(-4,1){.}\put(0,4){.}
                          \put(4,7){.}\end{picture}}}

\newcommand {\eq} [1] {\begin{equation}\label{#1}}
\newcommand {\en} {\end{equation}}
\newcommand {\eqn}  {\begin{eqnarray}}
\newcommand {\enn} {\end{eqnarray}}
\newcommand {\bstar}  {\begin{eqnarray*}}
\newcommand {\estar} {\end{eqnarray*}}

\newcommand {\I} {{\bf I}}

\newcommand {\U} {{\bf U}}
\newcommand {\V} {{\bf V}}
\newcommand {\W} {{\bf W}}
\newcommand {\X} {{\bf X}}
\newcommand {\Y} {{\bf Y}}
\newcommand {\Z} {{\bf Z}}

\newcommand {\nrm} [1] {|\!| #1 |\!|}

\newcommand {\mat} [1] {\left[\begin{array}{#1}}
\newcommand {\rix}     {\end{array}\right]}
\newtheorem{example}           {Example}

\newcommand {\range} {\mathop{\rm range}\nolimits}

\newcommand {\wh} {\widehat}
\newcommand {\wt} {\widetilde}
\newcommand {\re} {\mathop{\rm Re}\nolimits}

\newcommand{\ii}    {{\imath}}

\begin{document}
\title{Eigenstructure perturbations for a class of Hamiltonian matrices and solutions of related Riccati inequalities}
\author{Volker Mehrmann \footnotemark[1]
\and Hongguo Xu  \footnotemark[2]
}
\date{}
\maketitle
\renewcommand{\thefootnote}{\fnsymbol{footnote}}
\footnotetext[1]
{Institut f\"ur Mathematik, MA 4-5,  TU Berlin, Str. des 17. Juni 136,
D-10623 Berlin, FRG.
\texttt{mehrmann@math.tu-berlin.de}.
}

\footnotetext[2]{ Department of Mathematics, University of Kansas, Lawrence, KS 44045, USA.
\texttt{feng@ku.edu} }
\renewcommand{\thefootnote}{\arabic{footnote}}
\begin{abstract}
The characterization of the solution set for a class of algebraic Riccati inequalities is studied. This class arises in the passivity analysis of linear time invariant control systems. Eigenvalue perturbation theory for the Hamiltonian matrix associated with the Riccati inequality is used to analyze the extremal points of the solution set.

\end{abstract}

\section{Introduction}
We study the {set of Hermitian solutions} of a class of algebraic Riccati inequalities (ARI)
\begin{equation}\label{ari}
F^H X+XF+XGX+K\le 0,
\end{equation}
as well as the  related
algebraic Riccati equations (ARE) of the form
\begin{equation}\label{are}
F^H X+XF+XGX+K=0.
\end{equation}
We consider the special class of problems with $0\le G=G^H\in\mathbb C^{n\times n}$, $0\le K=K^H\in\mathbb C^{n\times n}$, and $F\in\mathbb C^{n\times n}$, where $\mathbb C^{n\times m}$ denotes the set of complex $n\times m$ matrices and $G >  0$ ($G \geq 0$) denotes that $G$ is positive definite (semidefinite). {The problem is studied in the setting of complex coefficients, but all our results also hold similarly in the real case replacing, if necessary, transformations concerning spectral properties by their real  versions, like e.g. the Schur (or Jordan) form by the real Schur form (or real Jordan form).}

The main goal of this paper is to characterize the set of positive
definite solutions of (\ref{ari}) and (\ref{are}) and we do this by studying the perturbation theory for the eigenstructure of the associated \emph{Hamiltonian matrix}
\begin{equation} \label{hammatrix}
H=\mat{cc} F & G \\ -K & -F^H\rix,
\end{equation}
formed from the coefficients, which satisfies $(JH)^H=JH$ for
$J=\mat{cc}0&I\\-I&0\rix.$

Our study of the solvability of \eqref{ari} and \eqref{are}
is motivated by the passivity analysis of linear time-invariant systems
\begin{equation} \label{GenSys}
\begin{array}{c}
\dot{x}\ =\ A\,x + B\, u,\\[1mm]
y\ =\ C\,x + D\, u,
\end{array}
\end{equation}
with $A\in \mathbb C^{n\times n}$, $B\in\mathbb C^{n\times m}$, $C\in\mathbb C^{m\times n}$, and $D \in\mathbb C^{m\times m}$, as well as { complex functions state $x$, input $u$,  and output $y$.}

A system of the form (\ref{GenSys}) is called \emph{(impedance)
passive}, see e.g. \cite{Wil72b},  if there exists a continuously differentiable \emph{storage function}
${\mathcal H}:\mathbb C^{n}\rightarrow [0,\infty)$ such that the \emph{dissipation inequality}
\begin{equation}  \label{DissipIneq}
 {\mathcal H}(x(t_1))-{\mathcal H}(x(t_0)) \leq \int_{t_{0}}^{t_{1}} {\re({y(t)^H}u(t))}\ dt,
\end{equation}
holds for all admissible inputs $u$.

For {the special case of systems} of the form (\ref{GenSys}) that are \emph{minimal}, i.e. both \emph{controllable} and \emph{observable},
the classical theorem of \cite{Wil72b} characterizes passivity.

\begin{theorem}\label{Willems1972}
Assume that the system (\ref{GenSys}) is minimal.
The matrix inequality
\begin{equation} \label{LMIcond}
\left[\begin{array}{cr} {A^H}X +X A & XB -C^H\\[2mm]
B^H X -C & -(D+D^H)
\end{array}\right]\leq 0
\end{equation}
has a solution $X= X^H> 0$ if and only if (\ref{GenSys}) is a \emph{passive} system. In this  case ${\mathcal H}(x) = \frac12 {x^H} X x$ defines a storage function for (\ref{GenSys}) associated
with the supply rate ${\re({y^H}u)}$, satisfying (\ref{DissipIneq}).

Furthermore, for passive systems  there exist  maximum and minimum Hermitian solutions to (\ref{LMIcond}) satisfying $0<X_{-}\leq X_{+}$ {(in the Loewner half-order of Hermitian matrices)} such that
for all Hermitian solutions $X$ to (\ref{LMIcond}), $X_{-}\leq X\leq X_{+}$.
\end{theorem}

{When the system is generated from purely data based interpolation approaches or from model reduction methods, then it typically cannot be guaranteed that the resulting system is minimal or it may happen that the system is close to an unobservable or uncontrollable system, see \cite{BeaMX15_ppt,CheMH19_ppt}. Also the minimality is only a sufficient condition for the existence of a positive definite solution, see \cite{CheGH23} for a detailed characterization. We therefore do not assume minimality and study  the existence and computation of positive definite solutions under weaker assumptions.
}

If the matrix $D+D^H$ in \eqref{LMIcond} is positive definite, then (using Schur complements) $X$ is a positive definite solution of \eqref{LMIcond} if and only if it is a positive definite solution of \eqref{ari} with
\begin{equation}\label{defcoeff}
F:=A-B(D+D^H)^{-1}C^H,\ G:=B(D+D^H)^{-1}B^H,\ K:=C^H(D+D^H)^{-1}C.
\end{equation}
The same argument applies to the equality case  \eqref{are}.

The solvability of \eqref{ari} is well understood, see e.g. \cite{BoyEFB94},
and the same holds for
\eqref{are}, see e.g. \cite{LanR95}. However,  in most cases the study is motivated from linear quadratic optimal control theory \cite{Meh91}, where the quadratic term has a negative sign, while here we study the case where the sign is positive.

A second motivation {for our work} arises from energy based modeling of passive systems via \emph{port-Ha\-mil\-to\-ni\-an (pH)}  systems. See \cite{MehU23} 
for a detailed survey on pH systems covering also the case of descriptor systems.

A linear time invariant \emph{port-Hamiltonian (pH) system} has the form
\begin{equation} \label{pH}
 \begin{array}{rcl} \dot x  & = & (J-R)Q x + (\hat B-\hat P) u,\\
y&=& (\hat B+\hat P)^HQ x+Du,
\end{array}
\end{equation}
with $Q=Q^{H} >0$, $J=-J^H$,
$
0\leq W^H=W:=\left[ \begin{array}{lc} R & \hat P \\ \hat P^H & S \end{array} \right]
$, with $S=\frac12(D+D^H)$.

It is well-known, see e.g. \cite{BeaMV19,CheGH23}, that if $X=X^H>0$ is a solution of \eqref{are} then multiplication of  \eqref{GenSys} from the left by $X^{\frac 12}$ { (the principal square root of $X$)}
 and a change of variable $\xi = X^{\frac 12}x$ leads to the pH formulation
\begin{eqnarray}
\label{phrepresentation}
 \dot \xi &=& X^{\frac 12} AX^{-\frac 12} \xi + X^{\frac 12} B u = (J-R) \xi + (\hat B- \hat P)u,\\
y & = &  C X^{-\frac 12} \xi + Du=(\hat B+ \hat P)^H \xi + (S+N) u, \nonumber
\end{eqnarray}
with  $J=\frac 12 (X^{\frac 12} AX^{-\frac 12}-X^{-\frac 12} A^H X^{\frac 12})$,
$R=\frac 12 (X^{\frac 12} AX^{-\frac 12} +X^{-\frac 12} A^H X^{\frac 12})$,
$\hat B=\frac 12 ({X^{\frac 12}}  B+X^{-\frac 12}C^H)$,
$\hat P=\frac 12 (-{X^{\frac 12}}  B+X^{-\frac 12}C^H)$,
$S=\frac 12 (D+D^H)$, $N=\frac 12 (D-D^H)$, and $Q=I$.

In this way any positive definite solution $X$  of \eqref{are} leads to a pH representation of the system. This immediately raises the question which solution of \eqref{are} leads to a robust or even optimally robust representation \eqref{phrepresentation}. {This has been the topic of several recent papers, where partial solutions are presented,
see \cite{BanMNV20,GilMS18, MehV20}. To fully answer} this question, it is essential to characterize the set of positive definite solutions of \eqref{LMIcond} and, {in particular, the extremal solutions in this set, which is one of the goals of this paper.}

A classical approach to solve \eqref{are} is to compute $n$-dimensional \emph{Lagrangian invariant subspaces}  of the Hamiltonian matrix~\eqref{hammatrix}, i.e.
subspaces satisfying
\begin{equation} \label{lagrange}
\mat{cc} F & G \\ -K & -F^H\rix \mat{c} W_1 \\ W_2 \rix =\mat{c} W_1 \\ W_2 \rix {T_{11}},
\end{equation}
where ${T_{11}} \in \mathbb C^{n\times n}$ and $W_1^HW_2=W_2^HW_1$.

If $W_1$ is invertible then $X=W_2W_1^{-1}$ is a Hermitian solution of \eqref{are} and the different solutions can be characterized by the choice of the Lagrangian invariant subspace which depends on the selection of the spectrum of {$T_{11}$}, see \cite{FreMX02} for a complete characterization of the set of Lagrangian subspaces of $H$. This solution set is particularly complicated if the matrix $H$ has purely imaginary eigenvalues, because in this case the approach to solve \eqref{are} via the numerical calculation of the Lagrangian invariant subspaces as  in \eqref{lagrange} may not be feasible. It may happen that either Lagrangian invariant subspaces do not exist or the matrix $W_1$ is not invertible, see
\cite{FreMX02,LanR95,Meh91} for a detailed analysis under controllability, observability, stabilizability, and detectability conditions. An approach that can deal with possible limiting situations where eigenvalues of \eqref{hammatrix} are on the imaginary is the use of Newton's method, see \cite{Ben97,KanN94}.

In this paper we study the Riccati equations \eqref{are} and inequalities \eqref{ari} arising in the passivity analysis, where both $G$ and $K$ are positive semidefinite. We derive the perturbation theory for the 
invariant subspaces in \eqref{lagrange}, and extend in particular the perturbation theory of \cite{MehX08}.

The paper is organized as follows.
In Section~\ref{sec:prelim} we recall and extend some known results on the positive definiteness of solutions of the ARE \eqref{are} and the notation.

Section~\ref{sec:evpt} then studies the behavior of the eigenvalues of the Hamiltonian matrix that is obtained when the ARE is perturbed by a positive semidefinite constant term.

\section{Classical solution theory} \label{sec:prelim}
In this section we present the notation and recall some well-known   results on the solvability of \eqref{ari} and \eqref{are}.

A pair of matrices $A\in \mathbb C^{n\times n}$, $B\in \mathbb C^{n\times m}$
is called \emph{controllable} if $\rank \mat{cc} sI-A &B \rix=n$ for all $s\in \mathbb{C}$, and a pair of matrices $A\in \mathbb C^{n\times n}$, $C\in \mathbb C^{p\times n}$
is called \emph{observable} if $\rank\mat{c}sI-A\\ C\rix=n$ for all $s\in \mathbb{C}$.
For our analysis  it will be very important to identify controllable as well observable subspaces, because the solution of \eqref{ari}  and \eqref{are} will turn out only to be nonsingular and therefore positive definite in these subspaces.

These subspaces can be obtained via {a condensed form under unitary transformations due to \cite{Van79}}.
We {present the result and a proof} in our notation because it will be useful for our analysis.
\begin{proposition}\label{prop:unitarykalman}
For  matrices $F,G,K$ as in \eqref{defcoeff} there exists a unitary matrix $U$ such that
\begin{eqnarray*}
U^H FU&=&\mat{cc|c}\wt F_{11}&0&\wt F_{13}\\\wt F_{21}&\wt F_{22}&\wt F_{23}\\\hline
0&0&\wt F_{33}\rix=:\mat{cc}F_{11}&F_{12}\\0&F_{22}\rix\\
U^H G U&=&\mat{cc|c}\wt G_{11}&\wt G_{21}^H&0\\\wt G_{21}&\wt G_{22}&0\\\hline
0&0&0\rix=:\mat{cc}G_{11}&0\\0&0\rix\\
U^H KU&=&\mat{cc|c}\wt K_{11}&0&\wt K_{31}^H\\0&0&0\\\hline
\wt K_{31}&0&\wt K_{33}\rix=:\mat{cc}K_{11}&K_{21}^H\\ K_{21}& K_{22}\rix,
\end{eqnarray*}
where $(F_{11},G_{11})$ and $(\wt F_{11},\wt G_{11})$ are  controllable,
and $(\wt F_{11},\wt K_{11})$ is observable.
\end{proposition}
\begin{proof}
The proof that $(F_{11},G_{11})$ is controllable and  $(\wt F_{11},\wt K_{11})$ is observable follows from the computation of the staircase form of $\mat{cc}F&G\rix$ and $\mat{cc}F_{11}^H&K_{11}\rix$, see \cite{Van79}.

Since $G\ge 0$, so is $G_{11}$, which implies that $(\wt F_{11},\wt G_{11})$ is controllable. To see this, suppose that $(\wt F_{11},\wt G_{11})$ is not controllable. Then there exists a nonzero vector $x$ and a scalar $\lambda$ such that  $x^H(\lambda I-\wt F_{11})=0$ and
$x^H \wt G_{11}=0$. Then $x^H{\wt G_{21}^H}=0$ as well. If this was not the case, then for $y={\wt G_{21}} x$, we have
\[
{\mat{cc}x&0\\0&y \rix^H} G_{11}\mat{cc}x&0\\0&y\rix =\mat{cc}0&\nrm{y}_2^2\\\nrm{y}_2^2
&y^H \wt G_{22}y\rix,
\]
which is indefinite, contradicting  the positive semi-definiteness of $G_{11}$. Thus
\[
\mat{c}x\\0\rix^H\mat{cc}\lambda I-F_{11}&G_{11}\rix=0,
\]
which contradicts the controllability of $(F_{11},G_{11})$.
\end{proof}

\subsection{Solvability of algebraic Riccati equations}\label{sec:aresol}
In order to characterize the existence of positive definite (positive semidefinite) solutions of \eqref{ari} and \eqref{are}, we have to identify what happens in the uncontrollable and unobservable subspaces.

To study this, let $U$ be the unitary matrix in Proposition~\ref{prop:unitarykalman} and let $U^H XU=\mat{cc}X_{11}&X_{21}^H\\X_{21}&X_{22}\rix$. Multiplying  both (\ref{ari}) and (\ref{are}) from the left by $U^H$ and from the right by $U$, the Riccati operator on the left sides takes the form
$
\mat{cc}
\Phi_{11}&\Phi_{21}^H\\\Phi_{21}&\Phi_{22}
\rix
$,
with
\begin{eqnarray*}
\Phi_{11}&=&
F_{11}^H X_{11}+X_{11}F_{11}+X_{11}G_{11}X_{11}+K_{11},\\
\Phi_{21}&=&X_{21}(F_{11}+G_{11}X_{11})+F_{22}^HX_{21}+F_{12}^HX_{11}+K_{21},\\
\Phi_{22}&=&
F_{22}^H X_{22}+X_{22}F_{22}+F_{12}^HX_{21}^H+X_{21}F_{12}+X_{21} G_{11}X_{21}^H+K_{22}.
\end{eqnarray*}
%
{For (\ref{ari}) (or (\ref{are})) to have a positive definite solution $X$ it
is necessary that $\Phi_{11}\le 0$ (or $\Phi_{11}=0$) has a positive definite
solution $X_{11}$.} The latter is equivalent to the condition that {$\Psi\le 0$ (or
$\Psi=0$})
has a positive definite solution $Y$, where $\Psi$ is the \emph{dual Riccati
operator} defined by
\[
\Psi=F_{11}Y+YF_{11}^H+YK_{11}Y+G_{11}.
\]
Using the block structures of $F_{11},K_{11},G_{11}$ given in
Proposition~\ref{prop:unitarykalman}, and partitioning
\[
Y=\mat{cc}Y_{11}&Y_{21}^H\\Y_{21}&Y_{22}\rix,
\]
 one has
\[
\Psi=\mat{cc}\Psi_{11}&\Psi_{21}^H\\\Psi_{21}&\Psi_{22}\rix,
\]
where
\[
\Psi_{11}=\wt F_{11}Y_{11}+Y_{11}\wt F_{11}^H+Y_{11}\wt K_{11}Y_{11}+\wt G_{11}.
\]
Its dual operator is
\[
\wt\Phi_{11}=\wt F_{11}^H\wt X_{11}+\wt X_{11}\wt F_{11}+\wt X_{11}\wt G_{11}\wt X_{11}+\wt K_{11}.
\]
As shown in \cite{BeaMX15_ppt}, a necessary condition for (\ref{ari}) { (or (\ref{are}))}
to have a positive definite solution is that $\wt\Phi_{11}\le 0$ (or $\wt\Phi_{11}=0$)
has a positive definite solution. Recall that $(\wt F_{11},\wt G_{11})$ is controllable
and $(\wt F_{11},\wt K_{11})$ is observable.
In the following we therefore mainly consider the matrix triples $(F,G,K)$ where these properties hold.

The ARE (\ref{are}) is associated with the {\em Hamiltonian matrix} \eqref{hammatrix}. Now suppose that the columns of
 $W=\mat{c}W_1\\W_2\rix\in\mathbb C^{2n\times n}$ are orthonormal and span a \emph{Lagrangian invariant subspace}, i.e.
\begin{equation}\label{LIS}
HW=WT_{11},\quad W^HW=I,\quad W^H JW=0
 \end{equation}
 for some matrix $T_{11}\in\mathbb C^{n\times n}$.
Then
\[
Q=\mat{cc}W_1&-W_2\\W_2&W_1\rix
\]
is {\em unitary and symplectic}, i.e., $Q^H Q=I$ and $Q^H J Q=J$,  and
\begin{equation}\label{hsh}
Q^HHQ=\mat{cc}T_{11}&T_{12}\\0&-T_{11}^H\rix
\end{equation}
is a {\em Hamiltonian Schur form} of $H$. Note that if $H$ is real then $Q$ can be chosen to be real orthogonal symplectic so that  $H$  has a real Hamiltonian Schur form, see \cite{LinMX99}.
\begin{lemma}[\cite{BeaMX15_ppt,LanR95}]\label{lem1}
Suppose that     the  Hamiltonian matrix~\eqref{hammatrix}
has a Lagrangian invariant subspace, i.e., there exist square matrices
$W_1$, $W_2$, {$T_{11}$} such that \eqref{LIS}  holds.
If the pair $(F,G)$ is controllable then $W_1$ is invertible, and if $(F,K)$ is observable then  $W_2$ is invertible.
\end{lemma}

{In \cite{Wil72a} a characterization is given
when all solutions of \eqref{are}
are positive definite.
We slightly improve this result, present it in the complex case and give a proof in our notation.}

\begin{lemma}\label{lem2}
All solutions of the ARE (\ref{are}) are positive definite if and only if
the pair $(F,G)$ is controllable, the pair $(F,K)$ is observable,
and $F$ is asymptotically stable, i.e., every
eigenvalue of $F$ has a negative real part.
 \end{lemma}
\begin{proof} Lemma~\ref{lem1} shows {that} the controllability of  $(F,G)$ and the observability of $(F,K)$  ensure the existence of invertible solutions and the asymptotic stability of $F$ guarantees the positive definiteness of the solutions. For necessity, without loss of generality we set
\[
F=\mat{cc}F_{11}&0\\ F_{21}&F_{22}\rix,\quad
G=\mat{cc}G_{11}&G_{21}^H\\
G_{21}&G_{22}\rix,\quad
K=\mat{cc}K_{11}&0\\0&0\rix,
\]
where $(F_{11},K_{11})$ is observable. If $X=\mat{cc}X_{11}&X_{21}^H\\X_{21}&X_{22}\rix$ is a positive definite solution, then $X_{22}$ is positive definite as well. With
\[
Y=\mat{cc}I&-X_{21}^HX_{22}^{-1}\\0&I\rix
\]
we then have
\[
\wt F =Y^{-H}FY^H=:\mat{cc}F_{11}&0\\\wt F_{21}&F_{22}\rix,\quad
\wt G=Y^{-H}GY^{-1}=:\mat{cc}G_{11}&\wt G_{21}^H\\\wt G_{21}&\wt G_{22}\rix
\]
and
\[
\wt X=YXY^H=\mat{cc}\wt X_{11}&0\\0&X_{22}\rix,\quad
\wt K=YKY^H=K.
\]
Since
$\wt F^H \wt X+\wt X\wt F+\wt X\wt G\wt X+\wt K=0$,
by comparing the blocks we have
\[
F_{11}^H \wt X_{11}+\wt X_{11}F_{11}+\wt X_{11}G_{11}\wt X_{11}+K_{11}=0
\]
and
\[
(\wt F_{21}^H +\wt X_{11}\wt G_{21}^H)X_{22}=0,\quad
F_{22}^H X_{22}+X_{22}F_{22}+X_{22}\wt G_{22}X_{22}=0.
\]
Using $\wt X_{11}$, we are able to construct $\wh X=\mat{cc}\wt X_{11}&0\\0&0\rix$
that solves (\ref{are}) as well. However, $\wh X$ is singular and hence  $(F,K)$
must be observable.

The matrix  $X$ is a Hermitian positive definite solution of (\ref{are}) if and only if $Y=X^{-1}$ is a Hermitian positive definite solution of
$FY+YF^H+YKY+G=0$.
We therefore conclude that $(F,G)$ must be controllable.

Finally, suppose that $\lambda$ is an eigenvalue of $F$ and $x$ is a corresponding
right eigenvector. Let $X$ be a positive definite solution of (\ref{are}). Then, since
\[
x^H(F^H X+XF+XGX+K)x=0,
\]
we have that
\[(\lambda+\bar\lambda) x^H Xx+x^H (X{G}X+K)x=0.
\]
Since $(F,K)$ is observable, we have that $Kx\ne 0$, and thus
\[
(\lambda+\bar\lambda) x^H Xx=-x^H(XGX+K)x\le -x^H Kx<0,
\]
which shows that $\lambda$ has  negative real part.
\end{proof}

The following result is well-known.
\begin{lemma}[\cite{BoyEFB94,LanR95}]\label{lem3} Suppose that $(F,G)$ is controllable and that $(F,K)$ is observable. {Then (\ref{are}) has a positive definite solution
if and only if (\ref{ari}) has a positive definite solution}.
\end{lemma}

In this subsection we have summarized some preliminary results about the existence of positive definite solutions of (\ref{are}) and (\ref{ari}). In the next subsection we discuss the construction of such solutions
for (\ref{are}) when $(F,G)$ is not controllable or $(F,K)$ is not observable.
\subsection{Positive definite solutions of the ARE (\ref{are})}
In this subsection we discuss the solution of the ARE \eqref{are} with  $F$ asymptotically stable and $G=G^H\geq 0,K=K^H\geq 0$.
{These results will be needed to analyze the subtle difference between the solutions of ARE~\eqref{are} and the ARI~\eqref{ari}.}
 We assume that the system has been transformed similarly as in Proposition~\ref{prop:unitarykalman} just with the roles
of $G$ and $K$ exchanged {and that $F$, $G$, $K$ are already in the following block forms}
\begin{eqnarray}
\nonumber
F&=&\mat{cc|c}\wt F_{11}&\wt F_{12}&0\\0&\wt F_{22}&0\\\hline
{\wt F_{31}}&\wt F_{32}&\wt F_{33}\rix=:\mat{cc}F_{11}&0\\ F_{21}&F_{22}\rix\\
\label{minf}
G&=&\mat{cc|c}\wt G_{11}&0&\wt G_{31}^H\\0&0&0\\\hline
\wt G_{31}&0&\wt G_{33}\rix=:\mat{cc}G_{11}&G_{21}^H\\ G_{21}&G_{22}\rix,\\
\nonumber
K&=&\mat{cc|c}\wt K_{11}&\wt K_{21}^H&0\\\wt K_{21}&\wt K_{22}&0\\\hline
0&0&0\rix=:\mat{cc}K_{11}&0\\0&0\rix,
\end{eqnarray}
where $(F_{11},K_{11})$ and $(\wt F_{11},\wt K_{11})$ are observable, and $(\wt F_{11},\wt G_{11})$ is controllable.

Suppose that the ARE
\begin{equation}\label{minare}
\wt F_{11}^H \wt X_{11}+\wt X_{11}\wt F_{11}+\wt X_{11}\wt G_{11}\wt X_{11}+\wt K_{11}=0
\end{equation}
has a solution $\wt X_{11}=\wt X_{11}^H$.
Since
$(\wt F_{11},\wt G_{11})$ is controllable, $(\wt F_{11},\wt K_{11})$ is observable, and
$\wt F_{11}$ is asymptotically stable,
by {Lemma~\ref{lem2} we have that} $\wt X_{11}$ is positive definite. Now we consider
the Sylvester equation
\[
(\wt F_{11}+\wt G_{11}\wt X_{11})^H \wt X_{21}^H+\wt X_{21}^H\wt F_{22}+\wt X_{11}\wt F_{12}
+\wt K_{21}^H=0.
\]
If $-(\wt F_{11}+\wt G_{11}\wt X_{11})^H$ and $\wt F_{22}$ do not have common
eigenvalues, then the Sylvester equation has a unique solution $\wt X_{21}^H$.
If $-(\wt F_{11}+\wt G_{11}\wt X_{11})^H$ and $\wt F_{22}$ share some common eigenvalues and the Sylvester equation is consistent, then we may choose a solution $\wt X_{21}^H$
(from infinitely many solutions). If it is not consistent, we need to choose a
$\wt X_{11}$ such that $-(\wt F_{11}+\wt G_{11}\wt X_{11})^H$ and $\wt F_{22}$
do not have the same eigenvalues. Note that $\wt F_{22}$ is asymptotically stable
and the eigenvalues of the Hamiltonian matrix $\mat{cc}\wt F_{11}&\wt G_{11}\\-\wt K_{11}&-\wt F_{11}^H\rix$ corresponding to (\ref{minare}) are those of $\wt F_{11}+\wt G_{11}\wt X_{11}$ and $-(\wt F_{11}+\wt G_{11}\wt X_{11})^H$. So it is always possible to
choose a solution $\wt X_{11}$ such that the common eigenvalues (with negative real parts) are contained in $\wt F_{11}+\wt G_{11}\wt X_{11}$, see~\cite{FreMX02}.
With such an $\wt X_{11}$ the Sylvester equation has a unique solution
{$\wt X_{21}^H$}.

Once we have determined $\wt X_{11}$ and $\wt X_{21}^H$, since $\wt F_{22}$ is asymptotically stable, the
Lyapunov equation
\[
\wt F_{22}^H \wt X_{22}+\wt X_{22}\wt F_{22}+\wt F_{12}^H \wt X_{21}^H+\wt X_{21} \wt F_{12}+\wt X_{21}\wt G_{11}\wt X_{21}^H+\wt K_{22}=0
\]
has a unique solution $\wt X_{22}$. Then the matrix $X_{11}=\mat{cc}\wt X_{11}&\wt X_{21}^H\\\wt X_{21}&\wt X_{22}\rix$ solves the ARE
\eq{minare2}
F_{11}^H X_{11}+X_{11}F_{11}+X_{11}G_{11}X_{11}+K_{11}=0.
\en
Since $(F_{11},K_{11})$ is observable, similar to the proof of Lemma~\ref{lem2} we can show
that $X_{11}>0$. Obviously, $\mat{cc}X_{11}&0\\0&0\rix$ is a positive semidefinite solution of
(\ref{are}) with $F,G,K$ given in (\ref{minf}). The ARE (\ref{are}) may or may not have positive definite solutions. Suppose it has one. Then we express it as
\begin{equation}\label{ares}
X=\mat{cc}I&Z\\0&I\rix\mat{cc}X_{11}&0\\0&X_{22}\rix
\mat{cc}I&0\\Z^H&I\rix,\qquad X_{11}>0,\quad X_{22}>0.
\end{equation}
By premultiplying $\mat{cc}I&Z\\0&I\rix$ and postmultiplying
$\mat{cc}I&0\\Z^H&I\rix$ to (\ref{are}), one has the ARE
\begin{eqnarray*}
&&\mat{cc}F_{11}&0\\\wh F_{21}&F_{22}\rix^H
\mat{cc}X_{11}&0\\0&X_{22}\rix+
\mat{cc}X_{11}&0\\0&X_{22}\rix
\mat{cc}F_{11}&0\\\wh F_{21}&F_{22}\rix\\
&&
+\mat{cc}X_{11}&0\\0&X_{22}\rix
\left(\mat{cc}I&0\\Z^H&I\rix G\mat{cc}I&Z\\0&I\rix\right)
\mat{cc}X_{11}&0\\0&X_{22}\rix+\mat{cc}K_{11}&0\\0&0\rix=0,
\end{eqnarray*}
where $\wh F_{21}=F_{21}+Z^HF_{11}-F_{22}Z^H.$
From the $(1,1)$ block, $X_{11}$ has to be a positive
definite solution of (\ref{minare2}), which always exists.
From the $(1,2)$ block, one has
\[
(\wh F_{21}^H+{X_{11}(G_{11}Z+G_{21}^H)})X_{22}=0.
\]
Because $X_{22}>0$, it becomes the Sylvester equation
\begin{equation}\label{1sylvester}
(F_{11}+G_{11}X_{11})^H Z-ZF_{22}^H+F_{21}^H+X_{11}G_{21}^H=0.
\end{equation}
From the $(2,2)$ block, $X_{22}$ solves the ARE
\begin{equation}\label{1lyapunov}
F_{22}^H X_{22}+X_{22}F_{22}+X_{22}(G_{22}+Z^H G_{21}^H+G_{21} Z+Z^H G_{11}Z)X_{22}=0.
\end{equation}
Since
\[
F_{11}+G_{11}X_{11}=\mat{cc}\wt F_{11}+\wt G_{11}\wt X_{11}&\wt F_{12}+\wt G_{11}\wt X_{21}^H\\
0&\wt F_{22}\rix,
\]
when $F_{11}+G_{11}X_{11}$ and $F_{22}$ do not share the same eigenvalues, the Sylvester equation \eqref{1sylvester} has a unique solution $Z$. Otherwise, if the Sylvester equation is consistent, we may choose one
solution $Z$. If it is not consistent, then we have two cases.
If $\wt F_{22}$ and $F_{22}$ do not have common eigenvalues,
as discussed before we may choose another solution $X_{11}$ ($\wt X_{11}$) such that $F_{11}+G_{11}X_{11}$ and $F_{22}$
have no common eigenvalues and then (\ref{1sylvester}) has a unique solution $Z$.
If $\wt F_{22}$ and $F_{22}$ have the same eigenvalues, (\ref{1sylvester})
has no solution and then the ARE (\ref{are}) has no
solution as well.

When both $X_{11}$ and $Z$ exist, then the ARE \eqref{1lyapunov} has an invertible
solution $X_{22}$ if and only if $X_{22}^{-1}$ solves
\[
F_{22}Y_{22}+Y_{22}F_{22}^H+\mat{c}Z\\ I\rix^H G\mat{c}Z\\I\rix=0.
\]
Such a solution exists if and only if $\left(F_{22}, \mat{c}Z\\I\rix^H G\mat{c}Z\\I\rix\right)$ is
controllable and in this case  $X_{22}>0$. When we have determined all
$X_{11}$, $Z$, and $X_{22}$, then $X$ given in (\ref{ares})
solves  (\ref{are}) with $(F,G,K)$ transformed as  in (\ref{minf}). Reversing the transformation to the form \eqref{minf} then gives a solution to the original ARE \eqref{are}.

Note that under the same assumptions, that $F$ is asymptotically stable
and that (\ref{minare}) has a positive definite solution, the
{Riccati} inequality (\ref{ari})
always has a positive definite solution
{regardless of the controllability  of $(F,G)$ and the observability of $(F,K)$, see \cite{BeaMX15_ppt}.
However, as we have demonstrated this may not be the case for the ARE (\ref{are}).
This shows the difference between (\ref{are}) and (\ref{ari}).
Consider the following example.}
{
\begin{example}\label{exam1}\rm Let
\[
F=\mat{cc|c}-2&0&0\\0&-1&0\\\hline 1&1&-1\rix,\quad
G=\mat{cc|c}1&0&0\\0&0&0\\\hline 0&0&1\rix,\quad
K=\mat{cc|c}3&1&0\\1&1&0\\\hline 0&0&0\rix.
\]
One can verify that \eqref{minare2} has the {unique} solution $X_{11}=\mat{cc}1&1/2\\1/2&5/8\rix$,
and \eqref{1sylvester} is inconsistent. So the corresponding ARE (\ref{are}) does not have a solution.
On the other hand the ARI \eqref{ari} has the solution
\[
X=\mat{ccc}3&1&-1\\1&1&0\\-1&0&1\rix\]
 that satisfies $F^HX+XF+XGX+K=\diag(-1,0,0)\le 0$.
\end{example}
}

{For a construction of solutions of (\ref{ari}) via numerical methods see \cite{BeaMX15_ppt}. However, since the solution of (\ref{ari}) is in general not unique, it is important to analyze the solution set and, in particular, the extremal solutions of (\ref{are}) and (\ref{ari}). This topic is studied in the next subsection.}

\subsection{Extreme solutions}\label{sec:extreme}
In this subsection we discuss the {extremal solutions of \eqref{are} and \eqref{ari}, i.e. the solutions } on the boundary of set of  positive definite solutions of \eqref{ari}.

If $X=X^H$ solves (\ref{are}), then
\begin{equation}\label{hsd}
H\mat{cc}I&0\\X&I\rix=\mat{cc}I&0\\X&I\rix\mat{cc}F+GX&G\\0&-(F+GX)^H\rix,
\end{equation}
where $H$ is defined in (\ref{hammatrix}).
So the columns of $\mat{c}I\\ X\rix$ form a Lagrangian invariant subspace of $H$ corresponding to the eigenvalues of $F+GX$ or equivalently,  $H$  has a Hamiltonian Schur form
(\ref{hsh}). However, once $H$ has a Hamiltonian Schur form, then  it may have many. The whole
classification of the Hamiltonian Schur form and therefore the solutions of (\ref{are}) were presented under some further assumptions in \cite{Wil71} and the complete picture is given in \cite{FreMX02}.

The existence of a Hamiltonian Schur form (\ref{hsh}) implies that
the algebraic multiplicity of each purely imaginary eigenvalue
(including zero) of $H$ must be even and the eigenvalues that are not purely imaginary are symmetrically placed with respect to  the imaginary axis. Recall the following results from \cite{Wil71} summarized in one lemma.
\begin{lemma}\label{lem4}
Consider a triple of matrices $(F,G,K)$ with $G=G^H\geq 0$ and $K=K^H \geq 0$.
\begin{enumerate}
\item[a.]
 Suppose that $X_-$ is a solution of (\ref{are}) such that
$F+GX_-$ is asymptotically stable and suppose that  $X$ is another solution of (\ref{are}).
Then $X_-\le X$. Also, $X_-$ is a monotonically non-decreasing function when increasing $K$ in the Loewner ordering.
\item[b.]
Suppose that $X_+$ is a solution of (\ref{are}) such that  $-(F+GX_+)$ is asymptotically stable and suppose that $X$ is another solution of (\ref{are}).
Then $X_+\ge X$. Also, $X_+$ is monotonically non-increasing when increasing $K$ in the Loewner ordering.
\item [c.]
Suppose that (\ref{are}) has two extreme solutions $X_-$ and $X_+$ with
both $F+GX_-$ and $-(F+GX_+)$ asymptotically stable.
Then any solution $X$ of (\ref{are})
satisfies $X_-\le X\le X_+$.
\item [d.] Suppose that $(F,G)$ is controllable, $(F, K)$ is observable, and (\ref{are}) has a solution. Then
(\ref{are}) has two extreme solutions $X_-$ and $X_+$ with all the eigenvalues of
$F+GX_-$ and $-(F+GX_+)$ in the closed left half complex plane. Furthermore any solution $X$ of (\ref{ari})
satisfies $X_-\le X\le X_+$.
\end{enumerate}
\end{lemma}

 Let us discuss the case d. in Lemma~\ref{lem4}. If $X$ solves (\ref{ari}), then there exists $\Delta_K=\Delta_K^H\ge 0$ such that
\begin{equation}\label{areW}
F^H X+XF+XGX+K+\Delta_K=0.
\end{equation}
Because $(F,G)$ is  controllable and $(F,K+\Delta_K)$ is still observable, the ARE~\eqref{areW} has two extremal solutions bounding $X$.
By Lemma~\ref{lem4} these extremal solutions are also bounded by $X_-$ and $X_+$, the extremal solutions of (\ref{are}). Thus $X$ is bounded by $X_-$ and $X_+$ as well.

Consider now a perturbation matrix
\[
\Delta=\mat{cc}\Delta_G&\Delta_F\\ \Delta_F^H&\Delta_K\rix\ge 0
\]
and suppose that the ARE
\begin{equation}\label{mare}
(F+\Delta_F)^H X+X(F+\Delta_F)+X(G+\Delta_G)X+(K+\Delta_K)=0
\end{equation}
has a solution $X$. Then
\[
F^H X+XF+XGX+K=-\Delta_F^H X-X\Delta_F-X\Delta_GX-\Delta_K
=-\mat{c}X\\I\rix^H \Delta\mat{c}X\\I\rix
\le 0,
\]
which implies that $X$ is a solution of \eqref{ari}. On the other hand, if $X$
solves (\ref{ari}), then there exists $\Delta_K\ge 0$ such that
\begin{equation}\label{arie}
F^H X+XF+XGX+K=-\Delta_K=-\mat{c}X\\I\rix^H \mat{cc}0&0\\0&\Delta_K\rix\mat{c}X\\I\rix.
\end{equation}
Thus we immediately  have the following result.
\begin{lemma}\label{lem7} The ARI (\ref{ari}) has a solution $X$ if and only if
$X$ solves (\ref{mare}) for some $\Delta\ge 0$.
\end{lemma}

\medskip
In this section we have reviewed {and slightly modified} classical results on the existence of positive (semi)-definite solutions to \eqref{are} and \eqref{ari}. In order to proceed with
the characterization of the solution sets and the feasible set (of $\Delta$), we study in the next section
the eigenvalue perturbation theory for the Hamiltonian matrix \eqref{hammatrix}.

So far all the results are based on the assumption that (\ref{ari}),
(\ref{are}), or {their} subproblems have a positive definite solution. Conditions for the existence
of such a solution, in particular when the associated Hamiltonian matrix has purely eigenvalues, will be provided in the next section.

\section{Eigenvalue perturbation theory for Hamiltonian matrices}\label{sec:evpt}
{In this section we investigate how the perturbation $\Delta$
in \eqref{mare} or $\Delta_K$ in (\ref{areW}) affects the eigenvalues and invariant
subspaces of the Hamiltonian matrix \eqref{hammatrix}. We then give a complete
characterization of the region of $\Delta_K$ such that (\ref{areW}) has a solution, or equivalently,
(\ref{ari}) has a solution.}

We assume first that $(F,G)$ is controllable.
Suppose that (\ref{are}) has a solution ${X_0}$, or equivalently that  (\ref{hsd}) holds.
{We first investigate the situation when the spectrum of $F+G{X_0}$ or $H$} contains purely imaginary
(and zero) eigenvalues. In this case {(using the  Jordan form of $F+G X_0$)} there exists an invertible matrix $M$ such that
\[
M^{-1}(F+G{X_0}) M=\diag(T_1,T_2),
\]
where $T_1$ and $T_2$ contain the eigenvalues with nonzero real parts and zero real parts, respectively.
If we partition $M^{-1}GM^{-H}=[G_{ij}]$ conformably, then
\[
\mat{cc}M^{-1}&0\\0&M^H\rix\mat{cc}F+G{X_0}&G\\0&-(F+G{X_0})^H\rix\mat{cc}M&0\\0&M^{-H}\rix
=\mat{cc|cc}T_1&0&G_{11}&G_{21}^H\\
0&T_2&G_{21}&G_{22}\\\hline
0&0&-T_1^H&0\\ 0&0&0&-T_2^H\rix.
\]
Let $Z_{12}$ be the unique solution of the Sylvester equation
\[
T_1Z_{12}+Z_{12}T_2^H+G_{21}^H=0,
\]
and define the symplectic matrix
\[
Z=\mat{cc|cc}I&0&0&Z_{12}\\ 0&I&Z_{12}^H&0\\\hline 0&0&I&0\\ 0&0&0&I\rix.
\]
Then for $S=\mat{cc}M&0\\ 0&M^{-H}\rix Z$, one has
\begin{equation}\label{dsch}
S^{-1}\mat{cc}F+G{X_0}&G\\0&-(F+G{X_0})^H\rix S =
\mat{cc|cc}T_1&0&G_{11}&0\\
0&T_2&0&G_{22}\\\hline
0&0&-T_1^H&0\\ 0&0&0&-T_2^H\rix.
\end{equation}
In this way, we have decoupled the part that contains all the purely imaginary eigenvalues, which is the  block upper triangular Hamiltonian matrix
$\mat{cc}T_2&G_{22}\\0&-T_2^H\rix$. Since $G\ge 0$, one also has $G_{22}\ge 0$.

Since $(F,G)$ is controllable,
 $(T_2,G_{22})$ is controllable as well. To see this, suppose that $\ii\gamma$ ($\ii=\sqrt{-1}$) is an eigenvalue
of $T_2$
and $x$ is a
left eigenvector such that $x^H \mat{cc}T_2-\ii\gamma I&G_{22}\rix=0$.
Then it follows that ${x^HG_{21}}=0$. If this were not the case, then there
{would exist} a vector $x_1\ne 0$ such that ${x^HG_{21}}x_1\ne 0$ (say,
$x_1={G_{21}^H}x$). Then
\[
\mat{cc}x_1&0\\0&x\rix^H G\mat{cc}x_1&0\\0&x\rix
=\mat{cc}x_1^H G_{11}x_1&x_1^H G_{21}^Hx\\x^H G_{21}x_1 &0\rix
\]
{would be} indefinite, contradicting that $G\ge 0$. We then have
\begin{eqnarray*}
&& \mat{c}0\\x\rix^H M^{-1}{\mat{c|c}F+G{X_0}-\ii\gamma I&G\rix}
\mat{cc}M&0\\0&M^{-H}\rix \\
&&\qquad =\mat{c}0\\x\rix^H\mat{cc|cc}T_1-\ii\gamma I&0&G_{11}&G_{21}^H\\0&T_2-\ii\gamma I&G_{21}&G_{22}\rix=0,
\end{eqnarray*}
which shows that $(F+G{X_0},G)$ is not controllable. But this contradicts the
controllability  of $(F,G)$ due to the relation
\[
\mat{cc}F-\alpha I&G\rix\mat{cc}I&0\\{X_0}&I\rix=\mat{cc}F
+G{X_0}-\alpha I&G\rix
\]
for all $\alpha\in\mathbb C$. In the same way we can show that $(T_1,G_{11})$ is controllable.

{We summarize the observations in the following results.
\begin{theorem}\label{newthm1}
Suppose that the ARE (\ref{are}) has a solution $X_0$, that $(F,G)$ is controllable,
and that the Hamiltonian matrix $H$ corresponding to (\ref{are}) as in (\ref{hammatrix})
is similar to the right hand matrix in (\ref{dsch})
via a similarity transformation with $\mat{cc}I&0\\ X_0&I\rix S$,
where $T_1$ and $T_2$ contain the non-imaginary and purely imaginary eigenvalues,
respectively. Then
there exists a block upper triangular symplectic matrix
\[
\wt Z=\mat{cc}\wt Z_{11}&\wt Z_{12}\\0&\wt Z_{11}^{-H}\rix
\]
such that
\begin{equation}\label{sjord}
\wt Z^{-1} \mat{cc}T_2&G_{22}\\0&-T_2^H\rix \wt Z
=\mat{cc}\wt T&\wt G \\0&-\wt T^H\rix,
\end{equation}
where
\[
\wt T=\diag(N_1(\ii\gamma_1),\ldots,N_r(\ii\gamma_r)),\quad
\wt G =\diag(\wt G_1,\ldots,\wt G_r),
\]
$N_j(\ii\gamma_j)$ is a Jordan block associated with a purely imaginary eigenvalue
$\ii\gamma_j$ of  size $k_j\times k_j$, and $\wt G_j=e_{k_j}e_{k_j}^H$,
for $j=1,\ldots,r$. Here $e_{k_j}$ is the $k_j$th column of the identity
matrix. This implies that $\mat{cc}T_2&G_{22}\\0&-T_2^H\rix$ and $H$ have eigenvalues
$\ii\gamma_1,\ldots,\ii\gamma_r$ (not necessarily distinct) with $\ii\gamma_j$
having a Jordan block of size $2k_j\times 2k_j$ for $j=1,\ldots,r$.

In addition, let
$\mat{cc}U_{11}\\U_{21}\rix$ span an arbitrary Lagrangian invariant subspace of
$\mat{cc}T_1&G_{11}\\0&-T_1^H\rix$.
Then the  ARE \eqref{are} has a  corresponding solution
\[
X=X_0+M^{-H}\mat{cc}U_{21}U_{11}^{-1}&0\\0&0\rix M^{-1}.
\]
\end{theorem}

\begin{proof}
The decomposition (\ref{sjord}) follows from Theorem 4.2 in \cite{FreMX02} and the fact that $(T_2,G_{22})$
is controllable and $G_{22}\ge 0$. It also implies that the only Lagrangian invariant subspace of
the matrix $\mat{cc}T_2&G_{22}\\0&-T_2^H\rix$ is $\range\mat{c}I\\0\rix$.

Using (\ref{dsch}), any Lagrangian invariant subspace of $H$ can be expressed as the span of the matrix
\[
\mat{cc}I&0\\X_0&I\rix \mat{cc}M&0\\0&M^{-H}\rix
\mat{cc|cc}I&0&0&Z_{12}\\0&I&Z_{12}^H&0\\\hline 0&0&I&0\\ 0&0&0&I\rix
\mat{cc}U_{11}&0\\0&I\\\hline U_{21}&0\\0&0\rix.
\]
Since $(F,G)$ is controllable, by Lemma~\ref{lem4} it corresponds to an ARE solution, which is
$X$.
\end{proof}
}

Here, in particular, if ${X_0}=X_-$ is the solution such that  $F+GX_-$ has all its eigenvalues in the closed left half plane, then $U_{21}U_{11}^{-1}\ge 0$.
Similarly  if ${X_0}=X_+$ is a solution such that all eigenvalues of $F+GX_+$ are in the closed right half plane then $U_{21}U_{11}^{-1}\le 0$. This characterization is similar to the one in \cite{Wil71}.

We have the following corollary.
{
\begin{corollary}\label{newcor1}
If the Hamiltonian matrix $H$ has only purely imaginary
eigenvalues and the ARE \eqref{are} has a solution,
then the solution is unique and in this case $X_-=X_+$.
\end{corollary}
\begin{proof} This follows  from Theorem~\ref{newthm1} by using the fact that in this case the blocks $T_1$ as well
as $\mat{c}U_1\\U_2\rix$ are void.
\end{proof}

The analysis in this subsection shows that, from a perturbation theory point of view, when all the eigenvalues are converging to the imaginary axis  (by perturbing $G,K$) then all different solutions of \eqref{are} converge to a single solution.

The situation of general perturbations is discussed in the following subsections.}
\subsection{General perturbations}\label{subs1}
{In order to study general perturbations to the Hamiltonian matrix that are associated with the perturbation~\eqref{mare}, we first provide some preparatory results. For this we consider different subcases of perturbations, the perturbation of  purely imaginary eigenvalues and the perturbation of non-imaginary eigenvalues. After studying the perturbation theory for purely imaginary eigenvalues, we move non-imaginary eigenvalues to the imaginary axis and then freeze them for further perturbation and in this way we successively move all eigenvalues to the imaginary axis without destroying the solvability of the ARI.

To proceed like this,} we consider a Hamiltonian matrix that is already in the block upper triangular form
\begin{equation}\label{h0}
H_0=\mat{cc}F&G\\0&-F^H\rix,
\end{equation}
where $G=G^H\ge 0$ and $(F,G)$ is controllable. Let
$\Delta=\Delta^H=\mat{cc} \Delta_{11}& \Delta_{21}^H\\ \Delta_{21}&\Delta_{22}\rix\ge 0$
and consider the perturbed matrix $H_0+J\Delta $.

{We first show that if
 $(F, \Delta_{11})$ is not observable, then the unobservable part can be removed and we can consider a smaller problem}. Using a simultaneous congruence
transformation with a unitary matrix, $F$ and $\Delta_{11}$ can
be transformed to  the form
\eq{detfdel}
\mat{cc}F_{11}&F_{12}\\0&F_{22}\rix,\qquad  \mat{cc}0&0\\0& \wt\Delta_{22}\rix,
\en
where $(F_{22},\wt \Delta_{22})$ is observable. {This transformation is similar to the
forms for $F$ and $K$ in (\ref{minf}) but with the block rows and columns reversed.}
Without loss of generality we may assume that
both matrices are already in the given block form. Partition
\eq{partg}
G=\mat{cc}G_{11}&G_{21}^H\\ G_{21}&G_{22}\rix
\en
accordingly  and repartition
\eq{partdel}
\Delta=\mat{cc|cc}0&0&0&0\\0&\wt\Delta_{22}&\wt\Delta_{32}^H&\wt\Delta_{42}^H\\
\hline
0&\wt\Delta_{32}&\wt\Delta_{33}&\wt\Delta_{43}^H\\0&\wt\Delta_{42}&
\wt \Delta_{43}&\wt\Delta_{44}\rix
\en
conformably. The zero blocks
in $\Delta$ are a result of the block form of $\Delta_{11}$ and the fact that $\Delta\ge 0$. Note that the zero block structure also shows
that $(F+\Delta_{21},\Delta_{11})$ is not observable. Then
\[
H_0+J\Delta=\mat{cc|cc}F_{11}&F_{12}+\wt \Delta_{32}&G_{11}+\wt\Delta_{33}&(G_{21}+\wt \Delta_{43})^H\\
0&F_{22}+\wt \Delta_{42}&G_{21}+\wt \Delta_{43}&G_{22}+\wt \Delta_{44}\\
\hline
0&0&-F_{11}^H&0\\
0&-\wt\Delta_{22}&-(F_{12}+\wt\Delta_{32})^H&-(F_{22}+\wt \Delta_{42})^H\rix
\]
Clearly, the eigenvalues of $F_{11}$ and $-F_{11}^H$ are not affected by the perturbation matrix $\Delta$, the affected eigenvalues are those of the submatrix
\[
\mat{cc}
 F_{22}+\wt\Delta_{42}&G_{22}+\wt \Delta_{44}\\
-\wt \Delta_{22}&-(F_{22}+\wt \Delta_{42})^H\rix
=\mat{cc}F_{22}&G_{22}\\0&-F_{22}^H\rix+J\mat{cc}\wt\Delta_{22}&\wt\Delta_{42}^H
\\\wt\Delta_{42}&\wt\Delta_{44}\rix.
\]
Note that since $\mat{cc}\wt\Delta_{22}&\wt\Delta_{24}\\\wt\Delta_{24}^H&\wt\Delta_{44}\rix$ is a submatrix of $\Delta$, it is positive semidefinite.

Based on the previous analysis, we have the following observations.
\begin{enumerate}
\item[a.] For the perturbation analysis we may assume that $(F,\Delta_{11})$ is observable,
since otherwise we may work on a reduced size problem.
\item[b.] We are able to construct a perturbation {$\Delta$ with $(F,\Delta_{11})$ unobservable}
to perturb some eigenvalues
of $H_0$ while keeping the remaining eigenvalues unchanged.
\end{enumerate}

In order to simplify the perturbation theory we consider
\begin{equation}\label{perham}
H(t)=H_0+tJ\Delta,\qquad t\ge 0,
\end{equation}
We therefore assume that $(F, \Delta_{11})$ is observable
and that $(F,G)$ is controllable.

{Next,
we show that for sufficiently small $t>0$,  $H_0+tJ\Delta$ can then be decomposed (via a structured similarity transformation) as a direct sum
of two Hamiltonian matrices of  smaller sizes that contain different eigenvalues.
This helps to single out purely imaginary
eigenvalues which is necessary for the forthcoming analysis.
Note that in the case of real coefficients  this can be done by separating complex conjugate pairs of eigenvalues from the other eigenvalues.

We therefore assume without loss of generality that}
\begin{equation}\label{bab}
F=\mat{cc}F_{11}&0\\0&F_{22}\rix,\qquad G=\mat{cc}G_{11}&0\\0&G_{22}\rix,
\end{equation}
and that
$(\Lambda(F_{11})\cup\Lambda( -F_{11}^H))\cap(\Lambda(F_{22})\cup\Lambda(-F_{22}^H))=\emptyset$, where $\Lambda (F) $ denotes the spectrum of the matrix $F$. Partition
\begin{equation}\label{defe}
\Delta=\mat{cc|cc}\Delta_{11}&\Delta_{21}^H&\Delta_{31}^H&\Delta_{41}^H\\\Delta_{21}&\Delta_{22}&\Delta_{32}^H&\Delta_{42}^H\\\hline
\Delta_{31}&\Delta_{32}&\Delta_{33}&\Delta_{43}^H\\
\Delta_{41}&\Delta_{42}&\Delta_{43}&\Delta_{44}\rix
\end{equation}
conformably.

{ To apply the previous results, we }show that the  controllability of $(F,G)$ implies the controllability of
$(F_{11},G_{11})$ and $(F_{22},G_{22})$, and the observability of
$\left(F,\mat{cc}\Delta_{11}&\Delta_{21}^H\\\Delta_{21}&\Delta_{22}\rix\right)$ implies the observability of $(F_{11},\Delta_{11})$ and $(F_{22},\Delta_{22})$.

The controllability of $(F_{11},G_{11})$ and $(F_{22},G_{22})$
is obvious. For the observability statement, suppose that $(F_{11},\Delta_{11})$ is not observable.
Then there is an eigenvector $x\ne 0$ of $F_{11}$ such that $F_{11}x=\lambda x$ and
$\Delta_{11}x=0$. Due to the positive semidefiniteness of $\Delta $, one also has $\Delta_{21} x=0$.
Then
\[
F\mat{c}x\\0\rix=\lambda\mat{c}x\\0\rix,\quad \mat{cc}\Delta_{11}&\Delta_{21}^H\\\Delta_{21}&\Delta_{22}\rix
\mat{c}x\\0\rix=0,
\]
which is a contradiction. The observability of $(F_{22},\Delta_{22})$ follows in the same
way.

{Using the block structures in  \eqref{bab} and  \eqref{defe}, we see that
$H(t)$ defined in \eqref{perham} can be expressed as}
\[
 H(t)=\mat{cc|cc}F_{11}+t\Delta_{31}&t\Delta_{32}&G_{11}+t\Delta_{33}&t\Delta_{43}^H\\
t\Delta_{41}&F_{22}+t\Delta_{42}&t\Delta_{43}&G_{22}+t\Delta_{44}\\\hline
-t\Delta_{11}&-t\Delta_{21}^H&-(F_{11}+t\Delta_{31})^H&-t\Delta_{41}^H\\
-t\Delta_{21}&-t\Delta_{22}&-t\Delta_{32}^H&-(F_{22}+t\Delta_{42})^H\rix.
\]
Let $P$ be a permutation
matrix that interchanges the middle two block columns when it is postmultiplied to a matrix
of the block form as $H(t)$. Then
\[
\wt J=P^H J P=\mat{cc}J_1&0\\0&J_2\rix,
\]
where $J_1$ and $J_2$ are version of $J$ of smaller size, and
\begin{eqnarray*}
\wt H(t)&:=&P^H H(t) P=
\mat{cc|cc}F_{11}+t\Delta_{31}&G_{11}+t\Delta_{33}&t\Delta_{32}&
t\Delta_{43}^H\\
-t\Delta_{11}&-(F_{11}+t\Delta_{31})^H&-t\Delta_{21}^H&-t\Delta_{41}^H\\\hline
t\Delta_{41}&t\Delta_{43}&F_{22}+t\Delta_{42}&G_{22}+t\Delta_{44}\\
-t\Delta_{21}&-t\Delta_{32}^H&-t\Delta_{22}
&-(F_{22}+t\Delta_{42})^H\rix\\
&=:&\mat{c|c} H_1(t)&tJ_1 L^H J_2\\\hline tL&H_2(t)\rix.
\end{eqnarray*}
Note that both $H_1(t)$ and $H_2(t)$ are Hamiltonian matrices.

Consider the quadratic matrix equation
\[
H_2(t)Y(t)-Y(t)H_1(t)-tY(t)J_1L^H J_2Y(t)+tL=0.
\]
{Since $\Lambda(H_1(0))=\Lambda(F_{11})\cup\Lambda(-F_{11}^H)$ and
$\Lambda(H_2(0))=\Lambda(F_{22})\cup\Lambda(-F_{22}^H)$, by our assumptions we have
$\Lambda(H_1(0))\cap\Lambda(H_2(0))=\emptyset$. Due to the continuity of the eigenvalues as functions of $t$, it follows that when $t$ is sufficiently small, then
$\Lambda(H_1(t))\cap\Lambda(H_2(t))=\emptyset$ and then
this equation has a unique solution $Y(t)$ with $\nrm{Y(t)}=\mathcal O(t)$.} Furthermore,  the matrices
\[
I-J_1Y^H(t) J_2Y(t),\quad I-Y(t)J_1Y^H(t) J_2
\]
are skew-Hamiltonian, and have the form $I+\mathcal O(t^2)$.

Using the structured Jordan canonical form, both skew-Hamiltonian matrices
have a principal skew-Hamiltonian square root
$S_1(t),S_2(t) = I+\mathcal O(t^2)$, see  \cite{FasMMX99}, satisfying
\[
S_1^2(t)=J_1^H S_1^H(t) J_1 S_1(t)=I-J_1Y^H(t) J_2Y(t),\quad
S_2^2(t)=J_2^H S_2^H(t) J_2S_2(t)=I-Y(t)J_1Y^H(t) J_2.
\]
Defining
\[
S(t)=\mat{cc}I&J_1Y^H(t) J_2\\Y(t)&I\rix\mat{cc}S_1^{-1}(t)&0\\0&S_2^{-1}(t)\rix,
\]
one has
\[
S^{-1}(t)=\mat{cc}S_1^{-1}(t)&0\\0&S_2^{-1}(t)\rix\mat{cc}I&-J_1Y^H(t) J_2\\-Y(t)&I\rix,\quad
S^H (t) \wt J S(t)=\wt J ,
\]
so that $P^H S(t)P$ is symplectic.

Using $S(t)$ and the fact that  $H_1(t)$ and $H_2(t)$ are Hamiltonian, we can eliminate
 the blocks $tL$ and $tJ_1L^H J_2$ in $\wt H(t)$
via
\begin{eqnarray*}
S^{-1}(t)\wt H(t)S(t)
&=&\mat{cc}S_1(t)(H_1(t)+tJ_1 L^H J_2Y(t))S_1^{-1}(t)&0\\
0&S_2^{-1}(t)(H_2(t)-tY(t)J_1L^H J_2) S_2(t)\rix\\
&=&\mat{cc}\wt H_1(t)&0\\0&\wt H_2(t)\rix,
\end{eqnarray*}
and by construction, both $\wt H_1(t)$ and $\wt H_2(t)$ are still Hamiltonian. Because
$S^H(t)\wt JS(t)=\wt J$, we also have
\[
S^{-1}(t)\wt H(t)S(t)=\wt J^H S^H(t)\wt J \wt H(t)S(t),
\]
from which we get
\begin{eqnarray}
\nonumber
\wt H_1(t)&=&J_1^H S_1^{-H }(t)\mat{c}I\\Y(t)\rix^H \wt J\wt H(t)\mat{c}I\\Y(t)\rix S_1^{-1}(t)\\
\nonumber
&=&J_1^H S_1^{-H }(t)(J_1H_1(t)-t L^H J_{2}Y(t)-t Y^H(t)J_2^H L
+Y^H(t) J_2H_2(t)Y(t))S_1^{-1}(t),\\
\label{h1}
\wt H_2(t)&=&J_2^H S_2^{-H}(t)\mat{c}J_1Y^H(t) J_2\\I\rix^H\wt J{\wt H(t)}
\mat{c}J_1Y^H(t)J_2\\I\rix S_2^{-1}(t)\\
\nonumber
&=&{J_2^H} S_2^{-H }(t)J_2^H[H_2(t)J_2+tY(t)J_1L^H+tLJ_1^H Y^H(t)
+Y(t)H_1(t)J_1Y^H(t)]J_2S_2^{-1}(t).
\end{eqnarray}
The Hamiltonian structure of $\wt H_1(t)$, $\wt H_2(t)$ is obvious. 

In this subsection we have reformulated the
Hamiltonian matrix that is obtained after perturbing the Riccati inequality \eqref{ari}. In the next subsection we use this reformulation to study the perturbation theory for the eigenvalues.
\subsection{Perturbations of a purely imaginary eigenvalue}
In this subsection we begin our analysis of  the perturbation of the eigenvalues of $H_0+J\Delta$ with $H_0$ as in \eqref{h0} in the transformed  version \eqref{h1}.
We first consider the case where $F_{11}$ as in (\ref{bab}) has a
single purely imaginary eigenvalue that is different from the eigenvalues of  $F_{22}$.
We  consider how the eigenvalues of $F_{11}$ and $-F_{11}^H$ are perturbed
in $H(t)$ as in (\ref{perham}) with $\Delta=\Delta ^H\ge 0$ given in (\ref{defe}).
Again, we assume that $(F_{11},G_{11})$ is controllable and  $(F_{11},\Delta_{11})$
is observable.
{
\begin{lemma}\label{newlem1}
Suppose that $(F_{11},G_{11})$ is controllable and  $(F_{11},\Delta_{11})$
is observable, and assume that $F_{11}$ has $s_\rho$ Jordan blocks of the size
$\rho\times \rho$ corresponding to the single purely imaginary eigenvalue $\ii\alpha$
for $\rho=1,\ldots,k$. Then
for sufficiently small $t>0$, all the eigenvalues of $\wt H_1(t)$ defined in \eqref{h1}
are semisimple. For each $\rho$,  there are exactly $s_\rho$ purely imaginary eigenvalues
of the form
\[
\lambda_{i,+}^{(\rho)}(t)=\ii (\alpha+(t\gamma_i^{(\rho)})^{\frac1{2\rho}})+
\mathcal O(t^{\frac1{\rho}}),\quad \quad i=1,\ldots,s_\rho
\]
and exactly $s_\rho$ purely imaginary eigenvalues of the form
\[
\lambda_{i,-}^{(\rho)}(t)=\ii (\alpha-(t\gamma_i^{(\rho)})^{\frac1{2\rho}})+
\mathcal O(t^{\frac1{\rho}}),\quad \quad i=1,\ldots,s_\rho,
\]
where $\gamma_i^{(\rho)}>0$ for $i=1,\ldots,s_\rho$. Suppose that the
columns of full column rank matrices $V_+^{(\rho)}(t)$ and $V_-^{(\rho)}(t)$ span the
invariant subspaces of $\wt H_1(t)$ corresponding to $\{\lambda_{i,+}^{(\rho)}(t)\}$
and $\{\lambda_{i,-}^{(\rho)}(t)\}$, respectively. Then
\[
\ii (V_+^{(\rho)}(t))^HJ_1V_+^{(\rho)}(t)<0,\qquad \ii (V_-^{(\rho)}(t))^HJ_1V_-^{(\rho)}(t)>0.
\]
\end{lemma}
}
{
\begin{proof}
The proof is presented in the Appendix.
\end{proof}
}

{
Lemma~\ref{newlem1} shows when $t$ is sufficiently small, then
$\wt H_1(t)$ has exactly $\sum_{\rho=1}^k2s_\rho$
semisimple purely imaginary eigenvalues, half of which are above $\ii\alpha$
(on the imaginary axis) with the corresponding
structure inertia indices $-1$ and half are below $\ii\alpha$ with the
corresponding structure inertia indices $1$, e.g. \cite{FreMX02,MehX08}.
As a consequence, $\wt H_1(t)$ does not have a Lagrangian invariant
subspace when $t$ is sufficiently small but positive, see e.g., \cite{FreMX02,MehX08}.
}

{We remark that based on the proof, in $\Delta$ only the block $\Delta_{11}$ plays
a role, because the set $\{\gamma_i^{(\rho)},\,i=1,\ldots,s_\rho,\,\rho=1,\ldots,k\}$ is determined by $\Delta_{11}$
only. The other blocks can be set to zero.}

\subsection{Perturbation theory for pre-existing purely imaginary eigenvalues}
Suppose that a general Hamiltonian matrix $H$ has a semisimple purely imaginary eigenvalue
$i\alpha$ with algebraic multiplicity $r$. Let $V$ be a full column rank matrix
such that
\[
HV=\ii\alpha V,\quad W:=\imath V^H JV < 0.
\]
The inequality indicates that the structure inertia indices corresponding to $\ii\alpha$ are all $-1$, see, \cite{FreMX02,LinMX99}.

Let $\Delta\ge 0$ and consider the perturbed Hamiltonian matrix $H(t)=H+tJ\Delta $. Following Theorem 3.2 in
\cite{MehX08}, for sufficiently small $t$, $H(t)$ has $r$ semisimple purely imaginary eigenvalues as
\begin{equation}\label{perteig}
\lambda_j(t)=\imath (\alpha+\delta_jt)+\mathcal O(t^2),\qquad j=1,\ldots,r,
\end{equation}
where $\delta_1,\ldots,\delta_r$ are the eigenvalues of the pencil $\lambda W+V^H \Delta V$.
Since $W<0$ and $V^H \Delta V\ge 0$, all $\delta_1,\ldots,\delta_r\ge 0$. If
${V^H \Delta  V}\ne 0$,
at least the imaginary part of some of $\lambda_1(t),\ldots,\lambda_r(t)$ is increasing.
Based on the first order
perturbation theory, there is a matrix $V(t)=V+\mathcal O(t)$ such that
\begin{equation}\label{pertinv}
H(t)V(t)=V(t)\diag(\lambda_1(t),\ldots,\lambda_r(t)),
\end{equation}
and
\[
W(t):=\imath V(t)^H JV(t)=\imath V^H JV+\mathcal O(t)<0.
\]
Similarly, when $W>0$ for some full rank matrix $V$ as above, $H(t)$ has
$r$ eigenvalues as those in (\ref{perteig}) but $\delta_1,\ldots,\delta_r\le 0$.
Also in this case there is a matrix $V(t)=V+\mathcal O(t)$ satisfying (\ref{pertinv}), but the corresponding $W(t)>0$.

Combining these observations with the results in the previous subsection, we have the global picture. Suppose that $H_0$
already has purely imaginary eigenvalues. Consider a small perturbation $J\Delta $ with
$\Delta=\mat{cc} \Delta_{11}&\Delta_{21}^H\\ \Delta_{21} &
\Delta_{22}\rix\ge 0$. If $Fx=\ii\alpha x$ and $ \Delta_{11}x\ne 0$
at least for some  purely imaginary eigenvalue $\ii\alpha$ of $F$, then $H_0+J\Delta$
still will have purely imaginary eigenvalues that are semisimple. The eigenvalues with the maximum imaginary parts must have structure inertia indices $-1$ and those with the minimum imaginary
parts must have structure inertia indices $1$.  If we keep adding perturbations $J\Delta $
with $\Delta \ge 0$, this will keep purely imaginary eigenvalues on the imaginary axis. Those with
negative structure inertia indices will be nondecreasing and those with positive structure inertia
indices will be nonincreasing. Some purely imaginary eigenvalues with opposite structure inertia
indices may collide and leave the imaginary axis by adding $J\Delta $ with $\Delta>0$ (\cite{MehX08}). However, based on the properties shown above, the extreme purely imaginary eigenvalues never have such a chance and will stay on the imaginary axis and keep the structure inertia indices.

In summary, once $H_0$ has purely imaginary eigenvalues, then $H_0+J\Delta$ will not have a Lagrangian
invariant subspace for any $\Delta\ge 0$, unless the purely imaginary eigenvalues $\ii\alpha$ of $F$ are
\emph{non-detectable} in $(F, \Delta_{11})$, i.e. $\rank [\ii\alpha I-F,\Delta_{11}]<n$.
\subsection{Perturbation of  eigenvalues that are not purely imaginary}
{Suppose that we have constructed a Hamiltonian matrix $H_0$ in (\ref{h0})  from $H$ given in (\ref{hammatrix})
by the transformation given in (\ref{hsd}) with a Hermitian solution $X$ of (\ref{are}).
By Lemmas~\ref{lem1} and Lemma~\ref{lem4}, the ARE
\[
F^HX+XF+XGX=0
\]
has two extreme solutions $X_-$ and $X_+$ (not necessarily positive semidefinite)
so that $F+GX_-$ and $-(F+GX_+)$
are asymptotically stable and $X_-\le X_+$.

In this subsection, we will show that by adding a matrix $J\Delta$ to $H_0$, with a nonzero perturbation $\Delta\ge 0$,
the matrix $H(t)=H_0+tJ\Delta$ will eventually have purely imaginary eigenvalues (including maybe the eigenvalue zero) when $t>0$ is sufficiently large.

Define the {\em feasible region of perturbations} for
the ARI (\ref{ari}) as the collection of all positive semidefinite matrices $\Delta$ such that
the ARE (\ref{mare}) has a Hermitian solution $X$.

In the following, we restrict our perturbations to the form $\Delta =\mat{cc}\Delta_{11}&0\\0&0\rix\ge 0$  with nonzero
$\Delta_{11}$.
Then for sufficiently large $t>0$ the matrix $H(t)=H_0+tJ\Delta$ has no Lagrangian invariant
subspace or equivalently the ARE
\begin{equation}\label{delta11t}
F^HX+XF+XGX+t\Delta_{11}=0
\end{equation}
has no Hermitian solution. If this would not be the case,  then \eqref{delta11t} has a solution $X(t)$ for all $t>0$.
Then, using the submultiplicativity of the spectral norm, we have
\[
t\nrm{\Delta_{11}}_2\le 2\nrm{F}_2\nrm{X(t)}_2+\nrm{G}_2\nrm{X(t)}_2^2.
\]
Following Lemma~\ref{lem4}, $X_-\le X(t)\le X_+$.  Then from
\[
X_+-X_-=(X_+-X(t))+(X(t)-X_-)
\]
one has $\nrm{X_+-X(t)}_2\le \nrm{X_+-X_-}_2$.
This implies that
\[
\nrm{X(t)}_2\le \nrm{X_+}_2+\nrm{X_+-X(t)}_2\le \nrm{X_+}+\nrm{X_+-X_-}_2=:\beta.
\]
Then
\[
t\nrm{\Delta_{11}}_2\le 2\nrm{F}_2\beta+\nrm{G}_2\beta^2,
\]
which is clearly not the case when $t$ is sufficiently large.

So let
$\Delta_0\ge 0$ restricted to the form $\mat{cc}\Delta_{11}&0\\0&0\rix$ be in the feasible region of the ARI.
For any such matrix $\Delta_0$, if $H+J\Delta_0$
still has no purely imaginary eigenvalues, then we may increase the magnitude of the
perturbation in the direction of $\Delta_0$  by considering $H+tJ\Delta_0$ for $t>1$,
until for a certain $t_0>1$, some eigenvalues of $H+t_0J\Delta_0$ are on the imaginary axis.
The matrix $t_0\Delta_0$
then touches the boundary of the feasible region.

We then choose
another $\Delta_1\ge 0$ such that the purely imaginary eigenvalues of $H+t_0\Delta_0$
remain the same in $H+(t_0\Delta_0+t\Delta_1)$. The ARE corresponding to
$H+(t_0\Delta_0+t\Delta_1)$ has solutions for $0\le t\le t_1$ but $H+(t_0\Delta+t_1\Delta_1)$
has more imaginary eigenvalues. Clearly,  $t_0\Delta_0+t\Delta_1$ stays  on the boundary
of the feasible region for $0\le t\le t_1$.

We then choose a perturbation $\Delta_2$ in the same way and continue.
Eventually, after a finite number of steps,
we have a $\Delta_F\ge 0$ such that all the eigenvalues of $H+J\Delta_F$ are purely imaginary
and the corresponding ARE has a single solution, i.e., the two extremal solutions coincide.
The resulting $\Delta_F$ is a vertex of the feasible region and cannot be perturbed further
by adding any positive semidefinite matrices. Note, however, that  the feasible region may have
several vertices.

In this way we are able to characterize boundary vertices of the feasible region of the ARI (\ref{ari}). Note that along the
paths described above, the ARE (\ref{mare}) always has two extremal solutions.
The minimal solution is nondecreasing and the maximal solution is nonincreasing
(both in the Loewner ordering) starting from $X_-$ and $X_+$, respectively, and eventually
they become the same.

\medskip
Consider the following example to illustrate the perturbation process.

\begin{example}\label{exam2}\rm
Consider the ARI (\ref{ari}) with
\[
F=\mat{cc}-3&-1\\-1&-5\rix,\quad G=\mat{cc}1&0\\0&1\rix,\quad K=\mat{cc}6&8\\8&17\rix.
\]
We have $\Lambda(F) =\{-4\pm\sqrt{2}\}$, $G,K>0$. So $F$ is asymptotically stable,
$(F,G)$ is controllable and $(F,K)$ is observable.
For the corresponding Hamiltonian matrix $H$, one has $\Lambda(H)=\{-2,-3,2,3\}$.
The ARE (\ref{are}) has two extremal solutions $X_-=\mat{cc}1&1\\1&2\rix$
and $X_+=\mat{cc}5&1\\1&8\rix$.
The set of real matrices $\Delta=\mat{cc}a&c\\c&b\rix \ge 0$ with
\eq{exare}
F^TX+XF+XGX+K+\Delta =0
\en
having Hermitian solutions can be characterized by the eigenvalue locations
of the Hamiltonian matrix
\[
H_\Delta :=H+\mat{cc}0&0\\-\Delta&0\rix,
\]
which satisfies
\[
\mat{cc}I&0\\X_-&I\rix^{-1} H_\Delta\mat{cc}I&0\\X_-&I\rix=
\mat{cc|cc}-2&0&1&0\\0&-3&0&1\\\hline -a&-c&2&0\\-c&-b&0&3\rix.
\]
The eigenvalues of $H_\Delta$ satisfy
\[
\lambda^2=\frac12(13-a-b\pm\sqrt{(a-b+5)^2+4c^2}).
\]
The positive semidefiniteness requires $(a,b,c)$ in the cone described by
\[
 a\ge 0,\quad b\ge 0, \quad ab \ge c^2.
\]
All four eigenvalues are real of $H_\Delta$ when $(a,b,c)$ satisfies
\[
13-a-b\ge 0,\quad (a-4)(b-9)\ge c^2.
\]
The vertex of the cone $ (a-4)(b-9)\ge c^2$ is $(4,9,0)$, which is in the plane
$13-a-b=0$. So the two inequalities imply $(a,b,c)$ is inside the cone
and $a\le 4$ and $b\le 9$.
The feasible region is
\[
{\cal R}=\{(a,b,c)|0\le a\le 4,\quad 0\le b\le 9,\quad ab\ge c^2,\quad (a-4)(b-9)\ge c^2\}.
\]
See Figure~\ref{fig1}.
\begin{figure}[ht]
\caption{
 {\em Feasible region ${\cal R}$}
}\label{fig1}
\includegraphics[scale=0.45]{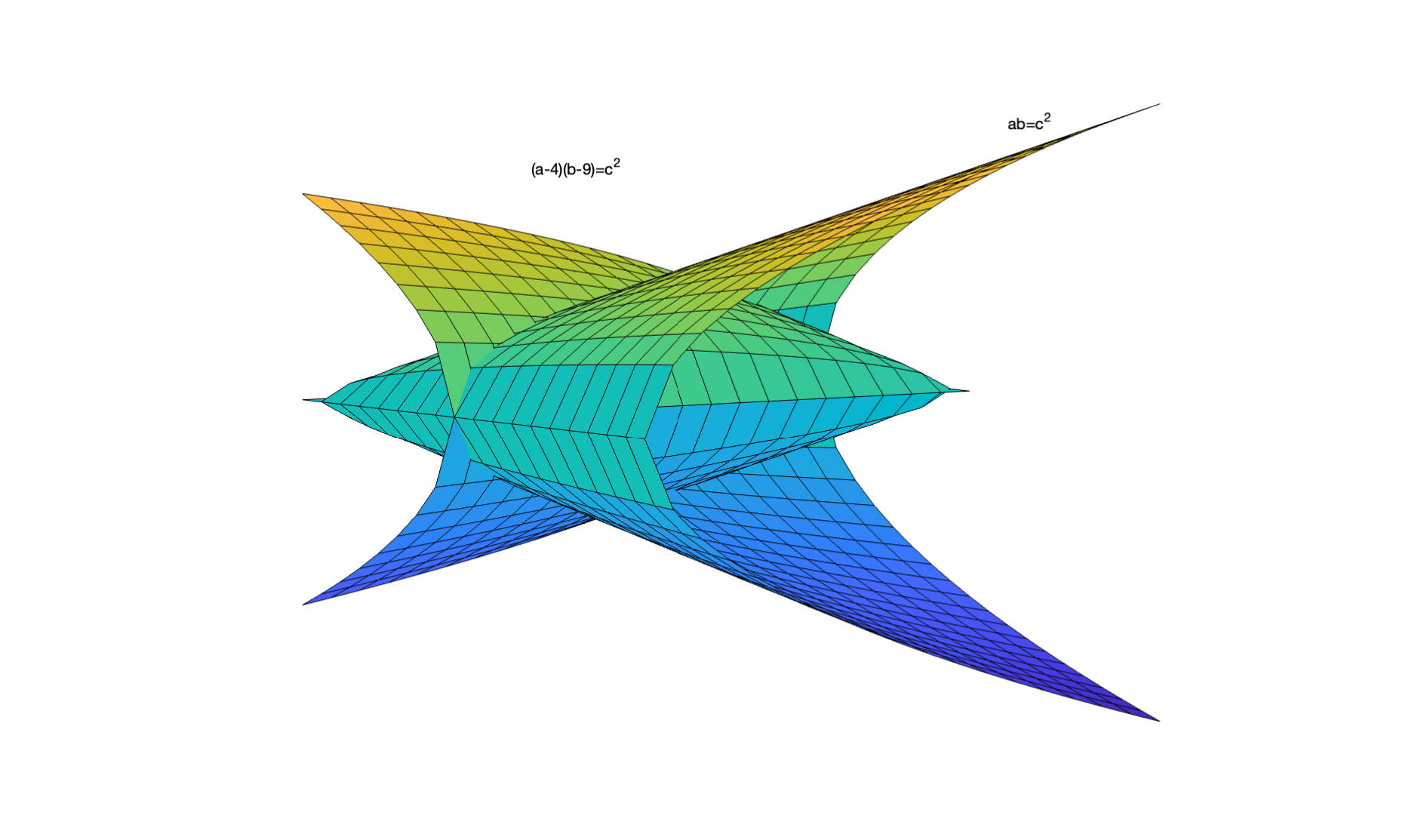}
\end{figure}
(Note the intersection of the two cones is on the plane $36-9a-4b=0$.) Then for any
$\Delta$ with $(a,b,c)\in{\cal R}$, the ARE (\ref{exare}) has Hermitian solutions.
The points in ${\cal R}$ satisfying $(a-4)(b-9)=c^2$ describe a surface.
The matrix $H_\Delta$ has exactly two zero eigenvalues for each $\Delta$
corresponding to such a point except for $(a,b,c)=(4,9,0)$. At this value $(4,9,0)$ all four eigenvalues
of $H_\Delta$ are zero, and the corresponding ARE has a unique solution $\mat{cc}3&1\\1&5\rix$.
\end{example}
}
\section{Conclusion}
We have analyzed the solution sets of the algebraic Riccati inequality
\eqref{ari} and the limiting algebraic Riccati equation~\eqref{are}.
Using eigenvalue perturbation theory we have shown how this set can be characterized via eigenvalue perturbation results for the associated Hamiltonian matrix.
The analysis shows that there is always a perturbation for the solution of the Riccati
inequality that leads to a unique solution where the associated Hamiltonian matrix has all eigenvalues on the imaginary axis and still has a unique Lagrangian invariant subspace.

{The same conclusions can be derived when all the matrices
are real, by paying attention to the  eigenstructure of  real Hamiltonian matrices
and considering  conjugate pairs of complex eigenvalues, e.g., \cite{LinMX99}.
}
{\section*{Acknowledgement}
We thank two anonymous referees and the handling editor for their suggestions that helped to improve the paper substantially.
}
\newpage

\section*{Appendix}
{\em Proof of Lemma~\ref{newlem1}.}
Following the analysis in Subsection~\ref{subs1}, the
perturbed eigenvalues are those of $\wt H_1(t)$ in (\ref{h1}).
Suppose that the single purely imaginary
eigenvalue of $F_{11}$ (as well as $H_1(0)=\mat{cc}F_{11}&G_{11}\\0&-F_{11}^H\rix$)
is $\ii\alpha$. Then $H_1(0)-\ii\alpha I$ is still Hamiltonian but with  single eigenvalue $0$.
Because of this, we may assume that zero is the only eigenvalue of $F_{11}$, i.e., $F_{11}$
is nilpotent.

As shown in \cite{LinMX99}, the matrix
$H_1(0)=\mat{cc}F_{11}&G_{11}\\0&-F_{11}^H\rix$ can be transformed to a
Hamiltonian Jordan form with a symplectic similarity transformation using a block
upper triangular matrix. Applying the transformation to $H_0+tJ\Delta$ on appropriate block rows and columns, the matrix $\Delta $ changes, but a close inspection shows that the controllability
and observability properties remain.
So as in (\ref{sjord}), without loss of generality, we may assume that
\[
F_{11}=\diag(\wh N_1,\ldots,\wh N_k),\quad G_{11}=\diag(G_1,\ldots,G_k),
\]
where, for $j=1,2,\ldots,k$,
\[
\wh N_j={\mat{cccc}0_{s_j}&I_{s_j}&&\\&\ddots&\ddots&\\&&\ddots&I_{s_j}\\
&&&0_{s_j}\rix_{j\times j}},\quad  {G_j=\mat{lcll}0_{s_j}&\ldots&0&0\\\vdots&\ddots&0&0\\
0&\ldots&0_{s_j}&0\\0&\ldots&0&I_{s_j}\rix_{j\times j}}.
\]
This implies that $F_{11}$ has $s_j$ Jordan blocks of order $j$
and $H_1(0)$ has $s_j$ Jordan blocks of order $2j$, for $j=1,\ldots,k$.
Also, the block forms show that $(F_{11},G_{11})$ is controllable.
Note that $\wh N_1$ is a zero matrix
of order $s_1$.
 If $s_j=0$ or $\wh N_j$ is void for some $j\in\{1,\ldots,k\}$, then $F_{11}$ does not have Jordan blocks of such an order. Let $P_0$ be a block permutation matrix such that
\[
P_0^H H_1(0)P_0=\diag\left(\mat{cc}\wh N_1&G_1\\
0&-\wh N_1^{H}\rix,\mat{cc}\wh N_2&G_2\\0&-\wh N_2^H\rix,
\ldots,\mat{cc}\wh N_k&G_k\\0&-\wh N_k^H\rix\right).
\]
With
\[\Sigma_j=\mat{ccccc}0&0&\ldots&0&(-1)^{j-1}I_{s_j}\\0&0&\ldots&(-1)^{j-2}I_{s_j}&0\\
\vdots&\vdots&\adots&\vdots&\vdots\\
0&-I_{s_j}&\ldots&0&0\\
I_{s_j}&0&\ldots&0&0\rix,\quad j=1,2,\ldots,k,
\]
and
\[
P=P_0\diag\left(\mat{cc}I_{s_1}&0\\0&\Sigma_1\rix,\mat{cc}I_{2s_2}&0\\0&\Sigma_2\rix,\ldots,
\mat{cc}I_{ks_k}&0\\0&\Sigma_k\rix\right),
\]
one has
\[
P^H H_1(0)P=\diag(N_{1},N_{2},\ldots,N_{k}),
\]
where
\[
N_{j}=\mat{cccc}0&I_{s_j}&&\\&\ddots&\ddots&\\&&\ddots&I_{s_j}\\
&&&0\rix_{2j\times 2j}
,\quad j=1,\ldots,k.
\]
With this $P$ one has
$
P^H J_1P=\diag\left(\wh \Sigma_1,\wh \Sigma_2,\ldots,\wh \Sigma_k\right)
$,
where
\begin{small}
\[
\wh \Sigma_j=
\mat{cccccc}0&\ldots&0&0&\ldots&(-1)^{j-1}I_{s_j}\\
\vdots&\vdots&\vdots&\vdots&\adots&\vdots\\
0&\ldots&0&I_{s_j}&\ldots&0\\
0&\ldots&-I_{s_j}&0&\ldots&0\\
\vdots&\adots&\vdots&\vdots&\vdots&\vdots\\
(-1)^{j}I_{s_j}&\vdots&0&0&\ldots&0\rix
=(-1)^j\mat{cccc}0&0&\ldots&(-1)^{2j-1}I_{s_j}\\
\vdots&\vdots&\adots&\vdots\\
0&-I_{s_j}&\ldots&0\\
I_{s_j}&0&\ldots&0\rix,
\]
\end{small}
for $j=1,\ldots,k$.
Recalling  that
\[
H_1(t)=H_1(0)+t\mat{cc}\Delta_{31}&\Delta_{33}\\-\Delta_{11}&-\Delta_{31}^H\rix,
\]
we partition
\begin{eqnarray*}
\Delta_{11}&=&\mat{cccc}\Delta^{(11)}_{11}&(\Delta^{(11)}_{21})^H&\ldots&(\Delta^{(11)}_{k1})^H\\
\Delta_{21}^{(11)}&\Delta^{(11)}_{22}&\ldots&(\Delta^{(11)}_{k2})^H\\
\vdots&\vdots&\ddots&\vdots\\
\Delta^{(11)}_{k1}&\Delta^{(11)}_{k2}&\ldots&\Delta^{(11)}_{kk}\rix,\quad
\quad
\Delta_{31}=\mat{cccc}\Delta_{(11)}^{(31)}&\Delta^{(31)}_{12}&\ldots&\Delta^{(31)}_{1k}\\
\Delta_{21}^{(31)}&\Delta^{(31)}_{22}&\ldots&\Delta^{(31)}_{2k}\\
\vdots&\vdots&\ddots&\vdots\\
\Delta^{(31)}_{k1}&\Delta^{(31)}_{k2}&\ldots&\Delta^{(31)}_{kk}\rix,\\
\Delta_{33}&=&\mat{cccc}\Delta^{(33)}_{11}&(\Delta^{(33)}_{21})^H&\ldots&(\Delta^{(33)}_{k1})^H\\
\Delta_{21}^{(33)}&\Delta^{(33)}_{22}&\ldots&(\Delta^{(33)}_{k2})^H\\
\vdots&\vdots&\ddots&\vdots\\
\Delta^{(33)}_{k1}&\Delta^{(33)}_{k2}&\ldots&\Delta^{(33)}_{kk}\rix
\end{eqnarray*}
conformably with the block form of $F_{11}$. Then
\[
P^H H_1(t)P=P^H H_1(0)P+tP^H \mat{cc}\Delta_{31}&\Delta_{33}\\-\Delta_{11}&-\Delta_{31}^H\rix P=P^H H_1(0)P+t\wt \Delta,
\]
where
\begin{tiny}
\[\wt \Delta
=\mat{cc|cc|c|cc}\Delta^{(31)}_{11}&\Delta^{(33)}_{11}\Sigma_1&\Delta^{(31)}_{12}&(\Delta^{(33)}_{21})^H\Sigma_2&\ldots
&\Delta^{(31)}_{1k}&(\Delta^{(33)}_{k1})^H\Sigma_k\\
-\Sigma_1^H \Delta_{11}^{(11)}&-\Sigma_1^H (\Delta_{11}^{(31)})^H\Sigma_1&-\Sigma_1^H (\Delta_{21}^{(11)})^H&
-\Sigma_1^H (\Delta_{21}^{(31)})^H\Sigma_2&\ldots&-\Sigma_1^H (\Delta_{k1}^{(11)})^H&
-\Sigma_1^H (\Delta_{k1}^{(31)})^H\Sigma_k\\
\hline
\Delta^{(31)}_{21}&\Delta^{(33)}_{21}\Sigma_1&\Delta^{(31)}_{22}&\Delta^{(33)}_{22}\Sigma_2&\ldots
&\Delta^{(31)}_{2k}&(\Delta^{(33)}_{k2})^H\Sigma_k\\
-\Sigma_2^H \Delta_{21}^{(11)}&-\Sigma_2^H (\Delta_{12}^{(31)})^H\Sigma_1&-\Sigma_2^H \Delta_{22}^{(11)}&
-\Sigma_2^H (\Delta_{22}^{(31)})^H\Sigma_2&\ldots&-\Sigma_2^H (\Delta_{k2}^{(11)})^H&
-\Sigma_2^H (\Delta_{k2}^{(31)})^H\Sigma_k\\
\hline
\vdots&\vdots&\vdots&\vdots&\ddots&\vdots&\vdots\\
\hline
\Delta^{(31)}_{k1}&\Delta^{(33)}_{k1}\Sigma_1&\Delta^{(31)}_{k2}&\Delta^{(33)}_{k2}\Sigma_2&\ldots
&\Delta^{(31)}_{kk}&\Delta^{(33)}_{kk}\Sigma_k\\
-\Sigma_k^H \Delta_{k1}^{(11)}&-\Sigma_k^H (\Delta_{1k}^{(31)})^H\Sigma_1&-\Sigma_k^H \Delta_{k2}^{(11)}
&-\Sigma_k^H (\Delta_{2k}^{(31)})^H\Sigma_2&\ldots&-\Sigma_k^H \Delta_{kk}^{(11)}&
-\Sigma_k^H (\Delta_{kk}^{(31)})^H\Sigma_k\rix.
\]
\end{tiny}
Partition further
\[
\Delta_{ij}^{(11)}=\mat{cccc}\wh \Delta_{11}^{(ij)}&\wh \Delta_{12}^{(ij)}&\ldots&\wh \Delta_{1j}^{(ij)}\\
\wh \Delta_{21}^{(ij)}&\wh \Delta_{22}^{(ij)}&\ldots&\wh \Delta_{2j}^{(ij)}\\
\vdots&\vdots&\ddots&\vdots\\
\wh \Delta_{i1}^{(ij)}&\wh \Delta_{i2}^{(ij)}&\ldots&\wh \Delta_{ij}^{(ij)}\rix, \quad i,j=1,\ldots,k
\]
and note that $\wh \Delta_{11}^{(11)}=\Delta_{11}^{(11)}$. Recall also that $\Delta_{11}$ is Hermitian and thus $\wh \Delta_{ji}^{(\ell\ell)}=(\wh \Delta_{ij}^{(\ell\ell)})^H$ for  $i,j=1,\ldots,\ell$ and $\ell=1,\ldots,k$.

From the observability of $(F_{11},\Delta_{11})$ it follows
that
\[
\Pi:=\mat{cccc}\wh \Delta_{11}^{(11)}&(\wh \Delta_{11}^{(21)})^H&\ldots&(\wh \Delta_{11}^{(k1)})^H\\
\wh \Delta_{11}^{(21)}&\wh \Delta_{11}^{(22)}&\ldots&(\wh \Delta_{11}^{(k2)})^H\\
\vdots&\vdots&\ddots&\vdots\\
\wh \Delta_{11}^{(k1)}&\wh \Delta_{11}^{(k2)}&\ldots&\wh \Delta_{11}^{(kk)}\rix >0.
\]
The submatrix of $\wt \Delta$ corresponding to the zero block rows and columns of
$P^H H_1(0)P$ is given by
\[
W=\mat{cccc}-\wh \Delta_{11}^{(11)}&-(\wh \Delta_{11}^{(21)})^H&\ldots&-(\wh \Delta_{11}^{(k1)})^H\\
\wh \Delta_{11}^{(21)}&\wh \Delta_{11}^{(22)}&\ldots&(\wh \Delta_{11}^{(k2)})^H\\
\vdots&\vdots&\ddots&\vdots\\
(-1)^k\wh \Delta_{11}^{(k1)}&(-1)^k\wh \Delta_{11}^{(k2)}&\ldots&(-1)^k\wh \Delta_{11}^{(kk)}\rix
=\diag(-I_{s_1},I_{s_2},\ldots,(-1)^kI_{s_k})\Pi.
\]
Define the  principal block submatrices
\[
\Pi_j=\mat{cccc}\wh \Delta_{11}^{(jj)}&(\wh \Delta_{11}^{(j,j+1)})^H&\ldots&(\wh \Delta_{11}^{(jk)})^H\\
\wh \Delta_{11}^{(j+1,j)}&\wh \Delta_{11}^{(j+1,j+2)}&\ldots&(\wh \Delta_{11}^{(k,j+1)})^H\\
\vdots&\vdots&\ddots&\vdots\\
\wh \Delta_{11}^{(kj)}&\wh \Delta_{11}^{(k,j+1)}&\ldots&\wh \Delta_{11}^{(kk)}\rix,
\]
and
\begin{eqnarray*}
W_j&=&\mat{cccc}(-1)^j\wh \Delta_{11}^{(jj)}&(-1)^j(\wh \Delta_{11}^{(j+1,j)})^H&\ldots&(-1)^j(\wh \Delta_{11}^{(kj)})^H\\
(-1)^{j+1}\wh \Delta_{11}^{(j+1,j)}&(-1)^{j+1}\wh \Delta_{11}^{(j+1,j+1)}&\ldots&
(-1)^{j+1}(\wh \Delta_{11}^{(k,j+1)})^H\\
\vdots&\vdots&\ddots&\vdots\\
(-1)^k\wh \Delta_{11}^{(kj)}&(-1)^k\wh \Delta_{11}^{(k,j+1)}&\ldots&(-1)^k\wh \Delta_{11}^{(kk)}\rix\\
&=&\mat{c|ccc}(-1)^j\wh \Delta_{11}^{(jj)}&(-1)^j(\wh \Delta_{11}^{(j+1,j)})^H&\ldots&(-1)^j(\wh \Delta_{11}^{(kj)})^H\\\hline
(-1)^{j+1}\wh \Delta_{11}^{(j+1,j)}&&&\\
\vdots&&W_{j+1}&\\
(-1)^k\wh \Delta_{11}^{(kj)}&&&\rix\\
&=&\diag((-1)^jI_{s_j},\ldots,(-1)^kI_{s_k})\Pi_j,
\end{eqnarray*}
for $j=1,\ldots,k$ with $\Pi_1=\Pi$, $W_1=W$, and $\Pi_k=\wh \Delta_{11}^{(kk)}$,
 $W_k=(-1)^{k}\wh \Delta_{11}^{(kk)}$. Then define the Schur complements
 $\mathcal S_k=W_k=(-1)^k\wh \Delta_{11}^{(k,k)}$,
 \begin{eqnarray*}
\mathcal S_j&=&(-1)^j\left(\wh \Delta_{11}^{(jj)}-\mat{c}
\wh \Delta_{11}^{(j+1,j)}\\\vdots\\\wh \Delta_{11}^{(kj)}\rix^H
W_{j+1}^{-1}\mat{c}(-1)^{j+1}\wh \Delta_{11}^{(j+1,j)}\\\vdots\\
(-1)^k\wh \Delta_{11}^{(kj)}\rix\right)\\
&=&(-1)^j\left(\wh \Delta_{11}^{(jj)}-\mat{c}
\wh \Delta_{11}^{(j+1,j)}\\\vdots\\\wh \Delta_{11}^{(kj)}\rix^H
\Pi_{j+1}^{-1}\mat{c}(\wh \Delta_{11}^{(j+1,j)}\\\vdots\\
\wh \Delta_{11}^{(kj)}\rix\right),\quad j=1,\ldots,k-1.
\end{eqnarray*}
Note that $(-1)^j\mathcal S_j>0$ for all $j=1,\ldots,k$. Then, from  the perturbation results in \cite{Lid66,MorBO97,Xu23a}, for a fixed
$\rho\in \{1,\ldots,k\}$, if $(-1)^\rho\gamma_{1}^{(\rho)},\ldots,(-1)^{\rho}\gamma_{s_\rho}^{(\rho)}$ are the $s_\rho$ eigenvalues of $\mathcal S_\rho$ (with $\gamma_i^{(\rho)}>0$ for $i=1,\ldots,s_\rho$),
then for $t$ sufficiently small, $P^H H_1(t)P$ as well as $H_1(t)$ has $2\rho s_\rho$ eigenvalues
\eq{eigdis}
\lambda^{(\rho)}_{ij}(t)=(t\gamma_i^{(\rho)})^{\frac1{2\rho}}\mu^{(\rho)}_{j}+\mathcal O(t^{\frac1{\rho}}),\quad
i=1,\ldots,s_\rho,\quad j=1,\ldots,2\rho,
\en
where $\mu^{(\rho)}_{1},\ldots,\mu^{(\rho)}_{2\rho}$ are the $2\rho$ roots of
$\mu^{2\rho}=(-1)^{\rho}$,
i.e.,
\[
\mu_{j}^{(\rho)}=\left\{\begin{array}{ll}e^{\frac{(2j-1)\pi \imath}{ 2\rho}}&\rho \mbox{ is odd}\\
e^{\frac{(j-1)\pi\imath}{\rho}}&\rho \mbox{ is even}\end{array}\right.
\qquad j=1,\ldots,2\rho.
\]
Note that for different values of $\rho$, the corresponding eigenvalue sets
are well separated due to the different fractional orders when $t$ is sufficiently small. For a fixed $\rho$,
since $\mathcal S_\rho$ is Hermitian, all the eigenvalues of $\mathcal S_\rho$ are semisimple.
Then all the eigenvalues {$\lambda^{(\rho)}_{ij}$} are semisimple as well, see \cite{Xu23a}.

The set
$\{\mu^{(\rho)}_{j}\}_{j=1}^{2\rho}$ contains exactly two purely imaginary numbers,
namely $\pm \imath$ for  $j = (\rho+1)/2,(3\rho+1)/2$ when $\rho$ is odd and
$j=  \rho/2+1,3\rho/2+1$ when $\rho$ is even. Let us denote the corresponding eigenvalues
of $H_1(t)$ by
\[
\lambda^{(\rho)}_{i,\pm}(t) =\pm \imath(t\gamma_{i}^{(\rho)})^{\frac{1}{ 2\rho}}+\mathcal O(t^{\frac{1}{ \rho}}),
\quad i=1,\ldots,s_\rho.
\]
Consider the eigenvalues $\lambda^{(\rho)}_{i,+}(t)$, $i=1,\ldots,s_\rho$.
Let
\[
{\Omega_{\rho,+}=\diag((\gamma_{1}^{(\rho)})^{\frac{1}{2\rho}},\ldots,
t\gamma_{s_\rho}^{(\rho)})^{\frac{1}{2\rho}})>0}.
\]
Following \cite{Xu23a} and recalling that the order of all Jordan blocks  of
$H_1(0)$ is even, there exists a full column rank matrix $V(t)=\mat{c}V_1(t)\\\vdots\\V_\rho(t)\\\vdots\\ V_k(t)\rix$ such that
\[
P^H H_1(t)PV(t)=V(t)(\imath\Omega_{\rho,+}(t)),\quad\mbox{or}\quad
 H_1(t)PV(t)=PV(t)(\imath\Omega_{\rho,+}(t)),
\]
where {$\Omega_{\rho,+}(t)=t^{\frac1{2\rho}}\Omega_{\rho,+}+\mathcal O(t^{\frac1{\rho}})$} and
\begin{eqnarray*}
V_i(t)&=&\mathcal O\left(\mat{c}t^{1-\frac{2i}{2\rho}}\\t^{1-\frac{2i-1}{2\rho}}\\\vdots\\t^{1-\frac{1}{2\rho}}\rix\right),\quad
i=1,\dots,\rho-1;\\
V_\rho(t)&=&
{\mat{c}Q\\\imath Qt^{\frac1{2\rho}}\Omega_{\rho,+}\\\vdots\\\imath^{2\rho-1} Qt^{\frac{2\rho-1}{2\rho}}\Omega_{\rho,+}^{2\rho-1}\rix}
+\mathcal O\left(\mat{c}t^{\frac1{2\rho}}\\t^{\frac2{2\rho}}\\\vdots\\t\rix\right),\\
V_i(t)&=&\mathcal O\left(\mat{c}t^{0}\\t^{\frac1{2\rho}}\\\vdots\\ t^{1-\frac1{2\rho}}\\t\\\vdots\\t\rix\right),
\quad i=\rho+1,\ldots,k,
\end{eqnarray*}
where $Q$ is a unitary matrix such that
$\mathcal S_\rho Q=(-1)^\rho Q\diag(\gamma_1^{(\rho)},\ldots,\gamma_{s_\rho}^{(\rho)})$, and the notation $\mathcal O(t^{\frac{j}{2\rho}})$ represents the fractional order (in $t$) of a block in the matrix.

Then
\[
(PV(t))^H J_1(PV(t))=\sum_{i=1}^kV_i(t)^H \wh\Sigma_i V_i(t),
\]
and one can easily check that
\[
V_i(t)^H\wh \Sigma_i V_i(t)=\mathcal O(t),\quad i=1,\ldots,\rho-1,\rho+1,\ldots,k,
\]
and
$
V_\rho(t)^H \wh\Sigma_\rho V_\rho(t)=
{2\rho\imath t^{\frac{2\rho-1}{2\rho}}\Omega_{\rho,+}^{2\rho-1}}+\mathcal O(t).
$
%
Hence,
\[
(PV(t))^H J_1(PV(t))={ 2\rho\imath t^{\frac{2\rho-1}{2\rho}}\Omega_{\rho,+}^{2\rho-1}}+\mathcal O(t).
\]
Then for sufficiently small $t>0$, by symmetry it follows that
\[
\imath (PV(t))^H J_1(PV(t))=-{2\rho t^{\frac{2\rho-1}{2\rho}}\Omega_{\rho,+}^{2\rho-1}}+\mathcal O(t)<0.
\]
Suppose that one of the eigenvalues $\lambda_{1,+}^{(\rho)}(t),\ldots,\lambda_{s_\rho,+}^{(\rho)}(t)$ is not purely imaginary, say $\lambda(t)$, with a corresponding right eigenvector $x(t)$, i.e.,
$H_1(t)x(t)=\lambda(t) x(t)$. By the symmetry of the spectrum we have
that $-\bar\lambda(t)$ is also an eigenvalue of $H_1(t)$ and, moreover,
$x(t)^HJ_1x(t)=0$. Apparently, if $\lambda(t)$ is a multiple eigenvalue,
due to the fractional order and eigenvalue distribution (\ref{eigdis}) they must be all in the set $\{\lambda_{j,+}^{(\rho)}(t)\}_{j=1}^{s_\rho}$. This implies
that $x(t)=PV(t)y(t)$ for some $y(t)\ne 0$. Then $x(t)^HJ_1x(t)=0$ implies that
\[
y(t)^H(PV(t))^TJ_1(PV(t))y(t)=0,
\]
contradicting the negative definiteness of $\ii (PV(t))^TJ_1(P(V(t))$.

We conclude that all $\lambda_{1,+}^{(\rho)}(t),\ldots,\lambda_{s_\rho,+}^{(\rho)}(t)$ must be
purely imaginary and the structure inertia index of $H_1(t)$ corresponding to each of them is $-1$. In the same way, one can show that all
$\lambda_{1,-}^{(\rho)}(t),\ldots,\lambda_{s_\rho,-}^{(\rho)}(t)$ are
also purely imaginary and
the structure inertia index
of $H_1(t)$ corresponding to each of them  is $+1$. It is easily verified that $\wt H_1(t)=H_1(t)+\mathcal O(t^2)$. Following \cite{Xu23a} and using the Hamiltonian structure of $\wt H_1(t)$
for each $\rho$, $\wt H_1(t)$ still has $2\rho$ purely imaginary eigenvalues
\[
\wt \lambda_{i,\pm}(t)
\pm\imath (t\gamma_i^{(\rho)})^{\frac1{2\rho}}+\mathcal O(t^{\frac1{\rho}}),
\quad i=1,\ldots,s_\rho,
\]
and their structure inertia indices are exactly the same as those of $\lambda_{i,\pm}(t)$,
$i=1,\ldots,s_\rho$. \qquad $\Box$


\begin{thebibliography}{10}

\bibitem{BanMNV20}
D.~Bankmann, V.~Mehrmann, Y.~Nesterov, and P.~Van Dooren.
\newblock Computation of the analytic center of the solution set of the linear
  matrix inequality arising in continuous- and discrete-time passivity
  analysis.
\newblock {\em Vietnam J. Mathematics}, 48:633--660, 2020.

\bibitem{BeaMV19}
C.~Beattie, V.~Mehrmann, and P.~{Van Dooren}.
\newblock Robust port-{H}amiltonian representations of passive systems.
\newblock {\em Automatica}, 100:182--186, 2019.

\bibitem{BeaMX15_ppt}
C.~Beattie, V.~Mehrmann, and H.~Xu.
\newblock Port-{H}amiltonian realizations of linear time invariant systems.
\newblock {\em http://arxiv.org/abs/2201.05355}, 2022.

\bibitem{Ben97}
P.~Benner.
\newblock Numerical solution of special algebraic {R}iccati equations via an
  exact line search method.
\newblock In {\em 1997 European Control Conference (ECC)}, pages 3136--3141.
  IEEE, 1997.

\bibitem{BoyEFB94}
S.~Boyd, L.~El Ghaoui, E.~Feron, and V.~Balakrishnan.
\newblock {\em Linear Matrix Inequalities in Systems and Control Theory}.
\newblock SIAM, Philadelphia, 1994.

\bibitem{CheGH23}
K.~Cherifi, H.~Gernandt, and D.~Hinsen.
\newblock The difference between port-{H}amiltonian, passive and positive real
  descriptor systems.
\newblock {\em Math. Control Signals Systems}, pages 1--32, 2023.

\bibitem{CheMH19_ppt}
K.~Cherifi, V.~Mehrmann, and K.~Hariche.
\newblock Numerical methods to compute a minimal realization of a
  port-hamiltonian system.
\newblock Technical report, Institute of Mathematics, TU Berlin, 2019.
\newblock http://arxiv.org/abs/1903.07042.

\bibitem{FasMMX99}
H.~Fa{\ss}bender, D.S. Mackey, N.~Mackey, and H.~Xu.
\newblock Hamiltonian square roots of skew-{H}amiltonian matrices.
\newblock {\em Linear Algebra Appl.}, 287:125--159, 1999.

\bibitem{FreMX02}
G.~Freiling, V.~Mehrmann, and H.~Xu.
\newblock Existence, uniqueness and parametrization of {L}agrangian invariant
  subspaces.
\newblock {\em {SIAM} J. Matrix Anal. Appl.}, 23:1045--1069, 2002.

\bibitem{GilMS18}
N.~Gillis, V.~Mehrmann, and P.~Sharma.
\newblock Computing nearest stable matrix pairs.
\newblock {\em Numer. Lin. Alg. Appl.}, 25:e2153, 2018.
\newblock https://doi.org/10.1002/nla.2153.

\bibitem{KanN94}
H.~Kano and T.~Nishimura.
\newblock Solution structure of algebraic matrix {R}iccati equations with
  nonnegative-definite quadratic and constant terms.
\newblock In {\em Proceedings of the ISCIE International Symposium on
  Stochastic Systems Theory and its Applications}, volume 1994, pages 43--48,
  1994.

\bibitem{LanR95}
P.~Lancaster and L.~Rodman.
\newblock {\em The Algebraic {R}iccati Equation}.
\newblock Oxford University Press, Oxford, 1995.

\bibitem{Lid66}
V.B. Lidskii.
\newblock Perturbation theory of non-conjugate operators.
\newblock {\em USSR Comput. Math. and Math. Physics}, 6:73--85, 1966.

\bibitem{LinMX99}
W.-W. Lin, V.~Mehrmann, and H.~Xu.
\newblock Canonical forms for {H}amiltonian and symplectic matrices and
  pencils.
\newblock {\em Linear Algebra Appl.}, 301--303:469--533, 1999.

\bibitem{Meh91}
V.~Mehrmann.
\newblock {\em The Autonomous Linear Quadratic Control Problem, Theory and
  Numerical Solution}, volume 163 of {\em Lecture Notes in Control and Inform.
  Sci.}
\newblock Springer-Verlag, Heidelberg, July 1991.

\bibitem{MehV20}
V.~Mehrmann and P.~Van Dooren.
\newblock Optimal robustness of port-{H}amiltonian systems.
\newblock {\em {SIAM} J. Matrix Anal. Appl.}, 41:134--151, 2020.

\bibitem{MehU23}
V.~Mehrmann and B.~Unger.
\newblock Control of port-{H}amiltonian differential-algebraic systems and
  applications.
\newblock {\em Acta Numerica}, pages 395--515, 2023.

\bibitem{MehX08}
V.~{Mehrmann} and H.~{Xu}.
\newblock Perturbation of purely imaginary eigenvalues of {H}amiltonian
  matrices under structured perturbations.
\newblock {\em Electron. J. Linear Algebra}, 17:234--257, 2008.

\bibitem{MorBO97}
J.~Moro, J.V. Burke, and M.L. Overton.
\newblock On the {L}idskii--{V}ishik--{L}yusternik perturbation theory for
  eigenvalues of matrices with arbitrary {J}ordan structure.
\newblock {\em {SIAM} J. Matrix Anal. Appl.}, 18:793--817, 1997.

\bibitem{Van79}
P.~{Van~Dooren}.
\newblock The computation of {K}ronecker's canonical form of a singular pencil.
\newblock {\em Linear Algebra Appl.}, 27:103--121, 1979.

\bibitem{Wil71}
J.~C. Willems.
\newblock Least squares stationary optimal control and the algebraic {R}iccati
  equation.
\newblock {\em IEEE Trans. Automat. Control}, AC-16(6):621--634, 1971.

\bibitem{Wil72a}
J.~C. Willems.
\newblock Dissipative dynamical systems -- {P}art {I}: {G}eneral theory.
\newblock {\em Arch. Ration. Mech. Anal.}, 45:321--351, 1972.

\bibitem{Wil72b}
J.~C. Willems.
\newblock Dissipative dynamical systems -- {P}art {II}: {L}inear systems with
  quadratic supply rates.
\newblock {\em Arch. Ration. Mech. Anal.}, 45:352--393, 1972.

\bibitem{Xu23a}
H.~Xu.
\newblock Invariant subspace perturbation of a matrix with {J}ordan blocks.
\newblock {\em https://arxiv.org/abs/2311.12219}, 2023.

\end{thebibliography}
\end{document}